%% file: main.tex
\documentclass{amsart}
\usepackage[T1]{fontenc}
\usepackage[utf8]{inputenc}
\usepackage[english]{babel}
\babelprovide[import]{swedish}
\usepackage{graphicx} 
\usepackage{amsmath}
\usepackage{array}
\usepackage{amssymb}
\usepackage{amsthm}
\usepackage{comment}
\usepackage[expansion=false]{microtype}
\usepackage{url}
\usepackage{tabularx}
\usepackage{hyperref}

\DeclareMathOperator{\tr}{tr}

\newtheorem{theorem}{Theorem}[section]

\newtheorem{proposition}[theorem]{Proposition}
\theoremstyle{definition}

\theoremstyle{remark}

\title[Weighted Laplace Spaces for Spectral Measures]{Weighted Laplace Spaces for Spectral Measures and Rational Approximation}
\author{Stefan Jakobsson}
\address{Colibrium Additive, GE Aerospace}
\email{stefan.jakobsson@geaerospace.com}
\author{Alice Kozakevicius}
\address{Skolförvaltningen, Mölndal stad, Sweden}
\email{alice.de-jesus-kozakevicius@molndal.se}
\author{Stig Larsson}
\address{Department of Mathematical Sciences, Chalmers University of Technology, Gothenburg, Sweden}
\email{stig.larsson@chalmers.se}
\date{\today}
\begin{document}
\begin{abstract}
    We introduce the weighted Laplace space $H_w$, an RKHS of Laplace transforms on $(0,\infty)$, and
    study spectral measures in its dual space $H'_w$.
    For conforming FEM discretizations of the Dirichlet Laplacian on bounded Lipschitz domains, we
    prove the dual-norm inequality $\|\mu_h\|_{H'_w} \leq \|\mu\|_{H'_w}$, where $\mu = \sum_k
        \delta_{\lambda_k}$ and $\mu_h = \sum_k \delta_{\lambda_{k,h}}$.
    The proof combines min-max monotonicity of FEM eigenvalues with a heat-trace representation of the
    dual norm.
    We then analyze $H_w$-adapted rational approximation of shifted symbols
    $\phi(x)=(x+\kappa^2)^{-\beta}$ and give a conditional transfer principle for estimates proved in
    the corresponding weighted Laplace pre-image norm.
    Via dual pairing, the norm inequality yields uniform bounds for finite spectral sums and related
    transformed observables.
\end{abstract}
\maketitle

\input{intro}
\input{declarationUseGenAI}
\input{fem}
\input{weightedTransformSpaces}

\input{weightedLaplaceSpaces}

\input{rational-approximation}

\input{matern-application}
\input{numerical-verification}
\input{conclusions}
\appendix
\input{appendix-rkhs-greens-functions}
\input{appendix-randomized-trace-estimation}
\input{appendix-rational-approximation-gradient}
\input{appendix-tensor-product-fem}
\bibliographystyle{plain}
\bibliography{refsSINUM}
\end{document}

%% file: intro.tex
\section{Introduction}
The spectral approximation of elliptic operators is a central topic in numerical analysis and
scientific computing \cite{BabuskaOsborn1991,Boffi2010}.
For a bounded Lipschitz domain $\Omega \subset \mathbb{R}^d$, let $\{\lambda_k\}_{k=1}^\infty$
denote the Dirichlet eigenvalues of the Laplacian and let $\{\lambda_{k,h}\}_{k=1}^{N_h}$ denote
the corresponding eigenvalues of a conforming finite element discretization.
We define the spectral measures
\begin{equation}
  \mu = \sum_{k=1}^\infty \delta_{\lambda_k}, \qquad
  \mu_h = \sum_{k=1}^{N_h} \delta_{\lambda_{k,h}}.
\end{equation}
In this work we study these spectral measures in a weighted dual Hilbert space framework adapted to
Laplace transforms.

We first formulate this framework abstractly through weighted transform spaces and their duals.
The weighted Laplace space $H_w$ is then obtained by specializing the transform to the Laplace
transform, while the spectral application enters through the heat trace of the measures above.
This provides a direct bridge from the general transform-side construction to the concrete norms
used for spectral comparison.

We introduce the weighted Laplace space $H_w$, a reproducing kernel Hilbert space of Laplace
transforms on $(0,\infty)$ with weight $w(t)$.
Under the relevant integrability conditions, its dual space $H'_w$ contains the continuous and
discrete spectral measures above.
The dual norm admits a heat-trace representation, which makes it particularly well suited for
comparing continuous and discrete spectra.

In particular, the duality pairing immediately yields a quantitative estimate
for finite spectral sums: for any $\psi\in H_w$,
\begin{equation}\label{eq:intro-duality-bound}
  \left|\sum_{k=1}^n \psi(\lambda_{k,h})\right|
  \leq \|\psi\|_{H_w}\,\|\mu_h\|_{H'_w}
  \leq \|\psi\|_{H_w}\,\|\mu\|_{H'_w}.
\end{equation}
This gives a concrete duality formulation of the comparison between discrete and continuous
spectra.
More importantly, \eqref{eq:intro-duality-bound} is the key transfer principle from
spectral-measure control to application-level error bounds: once $\|\mu_h\|_{H'_w}$ is controlled,
any observable represented by $\psi\in H_w$ for which the spectral pairing is defined is controlled
through the same dual pairing.
In this sense, the dual pairing inequality is central rather than merely technical.

The present paper is inspired by the rational SPDE approach of Bolin and Kirchner
\cite{BolinKirchner2020} and by the explicit connection between Gaussian fields and Gaussian Markov
random fields established by Lindgren, Rue and Lindström \cite{LindgrenRueLindstrom2011}.
This is exactly the type of inequality that matters in SPDE applications
\cite{LangLarssonSchwab2013,AnderssonLarsson2016,AnderssonKruseLarsson2016}, where one needs
combined estimates of the discrete spectra relative to the continuous model.

Our main contributions are:
\begin{itemize}
  \item formulation of weighted transform spaces and their specialization to
        weighted Laplace spaces $H_w$, together with the heat-trace dual norm representation;
  \item a dual-norm inequality for conforming FEM spectral measures,
        \begin{equation}\label{eq:intro-main}
          \|\mu_h\|_{H'_w} \leq \|\mu\|_{H'_w};
        \end{equation}
  \item transfer from spectral-measure control to uniform observable bounds via
        dual pairing;
  \item $H_w$-adapted rational approximation of shifted symbols, with explicit
        variable-projection optimization and conditional transfer of weighted exponential-sum estimates;
  \item numerical evidence that metric-adapted optimization improves finite-$n$
        accuracy at both norm and observable levels.
\end{itemize}

The eigenvalue monotonicity itself is classical; the novelty lies in the weighted dual-space
formulation, the observable-level transfer principle, and the unification of FEM spectral
comparison and rational approximation in one functional framework.

The inequality in \eqref{eq:intro-main} follows from Rayleigh--Ritz monotonicity and the identity
$\|\mu\|_{H'_w}^2=\int_0^\infty Z(t)^2w(t)\,dt$, where $Z(t)=\sum_{k=1}^\infty e^{-\lambda_k t}$ is
the heat trace.
Thus the discrete spectral measure has no more weighted heat-trace energy than the continuous one.
Combined with \eqref{eq:intro-duality-bound}, this gives uniform pairing bounds for finite spectral
sums and related transformed quantities.

Beyond spectral measures, we study approximation in the primal space $H_w$.
For shifted model functions of the form $\phi_\kappa(x) = (x+\kappa^2)^{-\beta}$, we show how
external approximation estimates transfer to optimal rational approximants in the $H_w$ norm.

Finally, Weyl's law gives the large-$\gamma$ asymptotic behavior of the continuous spectral measure
norm for weights $w(t)=t^\alpha e^{-\gamma t}$, with $\alpha>d-1$: \[ \|\mu\|_{H'_w}^2 \sim
  \frac{|\Omega|^2\Gamma(\alpha-d+1)}{(4\pi)^d\gamma^{\alpha-d+1}} \qquad \text{as }\gamma\to\infty.
\]
This is a small-time Weyl-law asymptotic.
In contrast, as $\gamma\downarrow0$, the exact heat-trace integral remains finite and converges to
the corresponding value with $w(t)=t^\alpha$, under the same integrability condition; the
large-time exponential decay of the heat trace controls this regime.

Taken together, these results establish a coherent theory linking finite element spectral
approximation, heat-trace energy, and rational approximation in weighted Laplace spaces.
Section~\ref{sec:weighted-transform-spaces} develops the abstract transform-side framework,
Section~\ref{sec:LaplaceSpaces} specializes it to weighted Laplace spaces and develops the
heat-trace dual-norm inequality for spectral measures.
Section~\ref{sec:rational-approx} develops the $H_w$ rational-approximation theory and optimization
formulation for shifted symbols, including conditional rate transfer, and
Section~\ref{sec:num-verification} provides numerical verification and application-facing
comparisons.

%% file: declarationUseGenAI.tex
\subsection{Declaration of generative AI and AI-assisted technologies in the manuscript preparation process}
During the preparation of this work, the author used \emph{GitHub Copilot} in order to 
\begin{itemize}
\item Develop ideas
\item Improve writing
\item Develop code
\item Develop tests
\end{itemize}
The tools are used in an assistive and iterative manner to support the author's work.
After using this tool/service, the author reviewed and edited the content as needed and takes full responsibility for the content of the published article.

%% file: fem.tex
\section{Conforming FEM Spectral Preliminaries}
\label{sec:fem-background}
Let $\Omega \subset \mathbb{R}^d$ be a bounded connected domain with Lipschitz boundary.
We consider the Dirichlet Laplace eigenvalue problem with homogeneous boundary conditions.
We seek eigenpairs $(\lambda_k,u_k)$ satisfying
\begin{equation}\label{eq:continuous-eigenproblem}
    \begin{aligned}
        -\Delta u_k & = \lambda_k u_k &  & \text{in } \Omega,         \\
        u_k         & = 0             &  & \text{on } \partial\Omega,
    \end{aligned}
\end{equation}
with eigenvalues
$0 < \lambda_1 < \lambda_2 \leq \cdots \to \infty$ and
$L^2$-orthonormal eigenfunctions $\{u_k\}_{k\ge 1}$.
For disconnected domains, the same arguments apply componentwise.

Let $V_h \subset H^1_0(\Omega)$ be a conforming finite element space of piecewise polynomials of
degree $p$ on a shape-regular mesh of diameter~$h$ of dimension $N_h$.
The
discrete eigenvalue problem is: find $(\lambda_{k,h},u_{k,h}) \in \mathbb{R}\times V_h$ such that

\begin{equation}\label{eq:discrete-eigenproblem}
    a(u_{k,h}, v_h) = \lambda_{k,h}\, (u_{k,h}, v_h) \qquad \forall v_h \in V_h,
\end{equation}
where the bilinear form is
\begin{equation}
    a(u,v) = \int_\Omega \nabla u \cdot \nabla v\, dx.
\end{equation}
The discrete problem produces eigenvalues $0 < \lambda_{1,h} \leq \lambda_{2,h} \leq \cdots \leq
    \lambda_{N_h,h}$.

Because $V_h \subset H^1_0(\Omega)$, the Rayleigh quotient (historically due to Rayleigh
\cite{Rayleigh1894}) is minimized over a smaller space.
This
yields the basic \emph{min-max monotonicity principle}
\cite{BabuskaOsborn1991,Boffi2010}:

\begin{equation}\label{minMaxLambda}
    \lambda_k \leq \lambda_{k,h} \qquad \text{for every } k=1,\dots,N_h.
\end{equation}

The remainder of this section surveys what can be said beyond \eqref{minMaxLambda}: upper bounds,
two-sided bounds, convergence rates, and higher-order asymptotic results.

By \emph{Weyl's law} (historically due to Weyl
\cite{Weyl1912}; see also \cite{BergerGauduchonMazet1971}), the continuous
eigenvalues satisfy
\begin{equation}\label{eq:weyl}
    \lambda_k \sim C_d\, \bigl(k / |\Omega|\bigr)^{2/d} \qquad \text{as } k \to \infty.
\end{equation}
Here \[ C_d = 4\pi^2 |B_1(0)|^{-2/d}, \] where $B_1(0) \subset \mathbb{R}^d$ is the unit ball.

Define the (continuous) heat trace
\begin{equation}\label{eq:heat-trace-fem}
    Z(t) := \sum_{k=1}^\infty e^{-\lambda_k t}, \qquad t>0.
\end{equation}
For each conforming discretization, define the discrete heat trace
\begin{equation}\label{eq:heat-trace-fem-discrete}
    Z_h(t) := \sum_{k=1}^{N_h} e^{-\lambda_{k,h} t}, \qquad t>0.
\end{equation}
\begin{proposition}[Heat-trace comparison]\label{prop:heat-trace-comparison}
    For every conforming discretization and every $t>0$,
    \begin{equation}\label{eq:heat-trace-comparison}
        0<Z_h(t)<Z(t).
    \end{equation}
\end{proposition}

\begin{proof}
    By min--max monotonicity and finiteness of $N_h$, \[ Z_h(t)=\sum_{k=1}^{N_h}e^{-\lambda_{k,h}t} \le
        \sum_{k=1}^{N_h}e^{-\lambda_kt} < \sum_{k=1}^{\infty}e^{-\lambda_kt} =Z(t).
    \]
    The strict inequality follows from the positive tail of the continuous heat
    trace, and positivity is immediate from the definitions.
\end{proof}

\begin{proposition}[Heat-trace asymptotics]\label{prop:heat-trace-asymptotics}
    The heat trace \eqref{eq:heat-trace-fem} satisfies
    \begin{equation}\label{eq:heat-trace-small-t}
        Z(t) \sim \frac{|\Omega|}{(4\pi t)^{d/2}} \qquad \text{as } t\to 0^+,
    \end{equation}
    and
    \begin{equation}\label{eq:heat-trace-large-t}
        Z(t) \sim e^{-\lambda_1 t} \qquad \text{as } t\to \infty.
    \end{equation}
\end{proposition}

\begin{proof}
    The small-time asymptotic \eqref{eq:heat-trace-small-t} is the classical Weyl/heat-kernel
    asymptotic (see \cite{BergerGauduchonMazet1971,Kac1966}).

    For large time, split \[ Z(t)=e^{-\lambda_1 t}+\sum_{k=2}^\infty e^{-\lambda_k t}.
    \]
    For $t\ge 1$,
    \[
        \sum_{k=2}^\infty e^{-\lambda_k t}
        =\sum_{k=2}^\infty e^{-\lambda_k(t-1)}e^{-\lambda_k}
        \le e^{-\lambda_2(t-1)}\sum_{k=2}^\infty e^{-\lambda_k}.
    \]
    By Weyl's law, the series $\sum_{k=2}^\infty e^{-\lambda_k}$
    converges.
    Hence \[ \frac{\sum_{k=2}^\infty e^{-\lambda_k t}}{e^{-\lambda_1 t}} \le
        C\,e^{-(\lambda_2-\lambda_1)t}\to 0, \] so \eqref{eq:heat-trace-large-t} follows.
\end{proof}

We also introduce the associated spectral measures on $(0,\infty)$:
\begin{equation}\label{eq:spectral-measures}
    \mu = \sum_{k=1}^\infty \delta_{\lambda_k}, \qquad
    \mu_h = \sum_{k=1}^{N_h} \delta_{\lambda_{k,h}},
\end{equation}
Then the heat traces are exactly the Laplace transforms of these measures: \[ Z(t)=\int_0^\infty
    e^{-\lambda t}\,d\mu(\lambda), \qquad Z_h(t)=\int_0^\infty e^{-\lambda t}\,d\mu_h(\lambda).
\]

The goal is to identify a Hilbert space $H$ of test functions on $(0,\infty)$
and sufficient conditions under which the discrete spectral measure is bounded
in the dual norm:

\begin{equation}\label{eq:star}
    \|\mu_h\|_{H'} \leq \|\mu\|_{H'}.
\end{equation}
Here $H'$ denotes the topological dual of $H$, with the $H'$-norm given by the operator norm of the
evaluation or integration functional.

For general background on spectral measures and self-adjoint spectral theory, see
\cite{ReedSimon1972,Weidmann1980}.

%% file: weightedTransformSpaces.tex
\section{Weighted Transform Spaces and Their Duals}\label{sec:weighted-transform-spaces}

This section develops the abstract framework behind the Laplace-space construction in
Section~\ref{sec:LaplaceSpaces}.
The goal is to isolate the mechanism that turns a weighted transform-side pairing into a
reproducing kernel Hilbert space (RKHS) and, in turn, into a dual norm.
The Laplace transform is the main example of interest, and the next section shows how it yields the
heat-trace norms used for spectral comparison.
For a brief review of RKHSs and the Green's-function interpretation of the weighted Laplace
kernels, see Appendix~\ref{app:rkhs-greens-function}.

We begin with signed Borel measures and a weighted transform-side pairing, and then identify the
associated RKHS structure.
All spaces are over the real numbers.
Let $\mathcal{M}$ be a linear space of signed Borel measures on $(0,\infty)$ containing the finite
linear combinations of Dirac masses.
Consider an injective integral transform
\begin{equation}
  \label{eq:AbstractT}
  (T\mu)(t) = \int_0^\infty T(t,s)\, d\mu(s).
\end{equation}
We assume that $T\mu$ is defined and continuous for every $\mu\in\mathcal{M}$ and that the kernel
$(t,\lambda)\mapsto T(t,\lambda)$ is measurable.
Let $w$ be measurable and positive almost
everywhere, and define
\begin{equation}\label{eq:weighted-feature-map}
  \Phi(\lambda)(t) = \sqrt{w(t)}\,T(t,\lambda).
\end{equation}
Assume that $\Phi(\lambda)\in L^2(0,\infty)$ for every $\lambda>0$.
Let $\mathcal{M}_w^0\subset\mathcal{M}$ be the space of measures for which $\Phi$ is Bochner
integrable with respect to $\mu$, and define
\begin{equation}\label{eq:feature-integral}
  A\mu = \int_0^\infty \Phi(\lambda)\,d\mu(\lambda) = \sqrt{w}\,T\mu \in L^2(0,\infty).
\end{equation}
The map $A$ is injective: if $A\mu=0$, then continuity of $T\mu$, positivity of $w$ almost
everywhere, and injectivity of $T$ successively give $T\mu=0$ and $\mu=0$.
We equip
$\mathcal{M}_w^0$ with
\begin{equation}\label{eq:dual-transform-inner-product}
  \langle \mu,\nu\rangle_{H_w'}
  =\langle T\mu,T\nu\rangle_{L^2(w)}
  :=\int_0^\infty (T\mu)(t)(T\nu)(t)w(t)\,dt
  =\langle A\mu,A\nu\rangle_{L^2},
\end{equation}
and define $H_w'$ as its Hilbert-space completion.
Define the synthesis transform
\begin{equation}\label{eq:weighted-synthesis-transform}
  (Sh)(\lambda)=\int_0^\infty T(t,\lambda)h(t)\,dt,
\end{equation}
and assume that it is well defined and injective on $L^2(w^{-1})$. Under the integrability
conditions that justify Fubini's theorem,
\begin{equation}\label{eq:formal-transpose-identity}
  \int_0^\infty (T\mu)(t)h(t)\,dt
  =\int_0^\infty (Sh)(\lambda)\,d\mu(\lambda).
\end{equation}
Thus $S$ is the formal transpose of $T$.
It is not denoted by $T^*$ because $\mathcal{M}$ has not been equipped with a Hilbert-space inner
product: the right-hand side of \eqref{eq:formal-transpose-identity} is the measure--function dual
pairing, not a Hilbert-space inner product on $\mathcal{M}$.

\begin{proposition}\label{prop:weighted-transform-rkhs}
  The kernel
  \begin{equation}\label{eq:ReproducingKernel}
    K(\lambda,\eta)=\langle\Phi(\lambda),\Phi(\eta)\rangle_{L^2}
    =\int_0^\infty T(t,\lambda)T(t,\eta)w(t)\,dt
  \end{equation}
  is strictly positive definite and determines an RKHS $H_w$. Moreover,
  \begin{equation}\label{eq:primal-inverse-transform-inner-product}
    \langle f,g\rangle_{H_w}
    =\langle S^{-1}f,S^{-1}g\rangle_{L^2(w^{-1})}.
  \end{equation}
  Moreover, the natural pairing
  \begin{equation}\label{eq:pairing}
    \langle\mu,f\rangle=\int_0^\infty f(\lambda)\,d\mu(\lambda)
  \end{equation}
  identifies $H_w'$ isometrically with the continuous dual of $H_w$.
\end{proposition}

\begin{proof}
  For distinct $\lambda_i$ and real $c_i$, \[ \sum_{i,j}c_ic_jK(\lambda_i,\lambda_j) =\left\|\sum_i
    c_i\Phi(\lambda_i)\right\|_{L^2}^2.
  \]
  Equality means $A(\sum_i c_i\delta_{\lambda_i})=0$, so injectivity of $A$ gives $c_i=0$ for
  every $i$.
  The Moore--Aronszajn theorem \cite{Aronszajn1950} therefore yields $H_w$.

  Let $G=\overline{\operatorname{span}}\{\Phi(\lambda):\lambda>0\}$ and define $Ug(\lambda)=\langle
    g,\Phi(\lambda)\rangle_{L^2}$.
  The map $U:G\to H_w$ is an isometry onto $H_w$.
  If $P_G$ denotes the orthogonal projection onto $G$, then $Sh=U(P_G(h/\sqrt w))$.
  Injectivity of $S$ implies $G=L^2(0,\infty)$, so $S^{-1}f=\sqrt w\,U^{-1}f$ and the asserted
  inner-product identity follows.

  For $\mu\in\mathcal{M}_w^0$, the Bochner integral $A\mu$ lies in $G$.
  If $f=Ug$, then \[ \int_0^\infty f(\lambda)\,d\mu(\lambda) =\langle g,A\mu\rangle_{L^2} =\langle
    f,U(A\mu)\rangle_{H_w}.
  \]
  Thus the pairing is bounded and its representing element has norm $\|A\mu\|_{L^2}$.
  For finite Dirac combinations these representers span the kernel sections, which are dense in
  $H_w$.
  Completion therefore identifies $H_w'$ isometrically with the full continuous dual of $H_w$.
\end{proof}

%% file: weightedLaplaceSpaces.tex
\section{Weighted Laplace Spaces for Spectral Measures}\label{sec:LaplaceSpaces}
In this section we introduce the weighted Laplace spaces $H_w$ associated with spectral measures, a
natural setting for problems in which a spectral measure is compared through its heat trace.
The preceding section gave the abstract weighted-transform framework; here we instantiate it for
the Laplace transform, which is the concrete operator relevant for our spectral comparison.
The resulting theory combines classical Bernstein--Stieltjes ideas on completely monotone functions
with modern reproducing kernel Hilbert space techniques (see
\cite{Widder1941,BergForst1975,Aronszajn1950}).
The connection to heat-trace asymptotics and spectral geometry is likewise classical (see
\cite{BergerGauduchonMazet1971,Kac1966}).
The specific formulation as a unified framework for spectral-measure comparison via weighted dual
norms is a synthesis of these classical elements adapted to the present problem.

\subsection{Laplace Measure Spaces}
We now turn from the abstract construction to the spectral application.
We specialize to the case where the operator $T$ is the Laplace transform, i.e., \[
  \mathcal{L}[\mu](t)=\int_0^\infty e^{-\lambda t}\,d\mu(\lambda).
\]
In this case the reproducing kernel becomes
\[
  K(\lambda,\eta)=\int_0^\infty e^{-(\lambda+\eta)t}\,w(t)\,dt
  =\mathcal{L}[w](\lambda+\eta),
\]
which is exactly the Laplace transform of the weight $w$ evaluated at $\lambda+\eta$.

Some classical weight functions and their resulting reproducing kernels are listed below.
\begin{table}[ht!]
  \centering
  \caption{Reproducing Kernels and Spaces for Given Weight Functions}
  \label{tab:kernels_spaces_experimantal}
  \begin{tabular}{|l|l|l|}
    \hline
    \textbf{Weight $w(t)$}    & \textbf{Kernel $K(\lambda,\eta)$}                & \textbf{Space $H_w$} \\ \hline
    $e^{-\alpha t}$           & $\dfrac{1}{\lambda+\eta+\alpha}$                 & Lorentz/Hardy–$H^2$  \\ \hline
    $t^{\nu-1} e^{-\alpha t}$ & $\dfrac{\Gamma(\nu)}{(\lambda+\eta+\alpha)^\nu}$ & Matérn-$\nu$         \\ \hline
    $e^{-t^2/(4\sigma^2)}$    & $e^{-\sigma^2(\lambda+\eta)^2}$ (Gaussian)       & Paley–Wiener         \\ \hline
  \end{tabular}
\end{table}

The names in the third column identify the classical RKHS family generated by the Laplace-transform
profile of each weight.
For $w(t)=e^{-\alpha t}$, one obtains the Cauchy-type kernel
$K(\lambda,\eta)=1/(\lambda+\eta+\alpha)$, corresponding to a Hardy/Lorentz half-plane structure.
For $w(t)=t^{\nu-1}e^{-\alpha t}$, the kernel
$K(\lambda,\eta)=\Gamma(\nu)/(\lambda+\eta+\alpha)^\nu$ has the fractional resolvent form
associated with Mat\'ern-type spaces.
For Gaussian-type weights, the resulting kernel is Gaussian in $(\lambda+\eta)$, which places the
corresponding RKHS in the Paley--Wiener/Gaussian-analytic class.

\subsection{The Laplace Spectral Measure Spaces}
This is the main result of the paper:
\begin{theorem}\label{thm:main-dual-norm}
  Let $w\colon(0,\infty) \to [0,\infty)$ be a measurable weight that is positive almost everywhere,
  and let $H_w'$ be the weighted Laplace space for that weight.
  Assume that the weight is such that $\mu, \mu_h \in H'_w$ and that the discrete and continuous
  eigenvalues fulfill the Rayleigh--Ritz monotonicity principle~\eqref{minMaxLambda}.
  Then
  \begin{equation}\label{eq:dual-norm-ineq}
    \|\mu_h\|_{H'_w} \leq \|\mu\|_{H'_w}.
  \end{equation}
\end{theorem}

\begin{proof}
  By Proposition~\ref{prop:heat-trace-comparison}, $0<Z_h(t)<Z(t)$ for every $t>0$.
  Since $w(t)\geq 0$,
  \begin{align*}
    \|\mu_h\|_{H'_w}^2
     & = \int_0^\infty Z_h(t)^2\, w(t)\, dt
    \leq \int_0^\infty Z(t)^2\, w(t)\, dt
    = \|\mu\|_{H'_w}^2.
  \end{align*}
  Taking square roots gives \eqref{eq:dual-norm-ineq}.
\end{proof}

\subsubsection{Sharpness and Gap Dependence}
By Proposition~\ref{prop:heat-trace-comparison}, the strict inequality $Z_h(t)<Z(t)$ holds for
every $t>0$.
Consequently, whenever $w$ is positive on a set of positive measure, \[ \|\mu_h\|_{H'_w}^2 =
  \int_0^\infty Z_h(t)^2 w(t)\,dt < \int_0^\infty Z(t)^2 w(t)\,dt = \|\mu\|_{H'_w}^2, \] so equality
in \eqref{eq:dual-norm-ineq} is impossible for any fixed mesh.

This strictness for fixed $h$ does not contradict sharpness of the constant $1$ in
\eqref{eq:dual-norm-ineq}.
Along any convergent refinement sequence, one still has $N_h<\infty$ for every $h$, but $Z_h(t)\to
  Z(t)$ for each $t>0$ and $0\le Z_h(t)\le Z(t)$.
Since $Z(t)^2w(t)$ is integrable whenever $\mu\in H'_w$, dominated convergence gives \[
  \|\mu_h\|_{H'_w}^2 \to \|\mu\|_{H'_w}^2 \qquad \text{as } h\to 0.
\]
Therefore no uniform constant $c<1$ can replace $1$ in
\eqref{eq:dual-norm-ineq}.

When the FEM spectrum approximates the continuous spectrum with the usual rate
$\lambda_{k,h}-\lambda_k = O(h^{2p})$, the gap in the squared norms is also controlled by the same
rate.

Indeed,

\[ \|\mu\|_{H'_w}^2 - \|\mu_h\|_{H'_w}^2 = \int_0^\infty
  \bigl(Z(t)+Z_h(t)\bigr)\bigl(Z(t)-Z_h(t)\bigr) w(t)\, dt.
\]

With $0 \le Z_h(t) \le Z(t)$ and the elementary bound $e^{-\lambda_k t} - e^{-\lambda_{k,h} t} \leq
  t e^{-\lambda_k t} (\lambda_{k,h}-\lambda_k)$, one obtains a bound of the form $O(h^{2p})$ for the
gap, provided the weight $w$ is compatible with the heat trace.
In particular, as the mesh is refined and the eigenvalue separation shrinks, the inequality becomes
increasingly tight, although it remains strict for every fixed $h$.

\subsection{Choosing weights}\label{subsec:choosing-weights}
We now choose the weight $w$ so that the spectral measure introduced in
eq.~\eqref{eq:spectral-measures} above belongs to $H_w'$.
The behavior of $w(t)$ near $t=0$ is tied to the high-frequency eigenvalue distribution through the
Weyl-scale singularity of $Z(t)$, whereas its behavior as $t\to\infty$ is tied to the low-lying
spectrum, in particular the decay governed by $\lambda_1$.
A natural choice is the Matérn-type weight from Table~\ref{tab:kernels_spaces_experimantal},
with a slightly different parametrization:
\begin{equation}\label{eq:spectral-weight}
  w(t) = t^{\alpha} e^{-\delta t},\qquad \alpha>-1, \quad \delta\geq 0.
\end{equation}
The following proposition follows from Proposition~\ref{prop:heat-trace-asymptotics} and the
definition of the weight.
\begin{proposition}\label{prop:dual-membership-weight}
  The spectral measure $\mu$ defined in \eqref{eq:spectral-measures} belongs to the dual space $H'_w$
  if and only if $\alpha>d-1$.
\end{proposition}

\subsection{Eigenvalue-Free Evaluation of \texorpdfstring{$\|\mu_h\|_{H'_w}$}{the discrete dual norm}}
\label{subsec:eigfree-evaluation}
For a conforming FEM discretization, let $K_h$ and $M_h$ denote the stiffness and mass matrices in
a chosen finite element basis.
The discrete eigenpairs are defined by the generalized eigenvalue problem
\begin{equation}
  K_h q_{k,h} = \lambda_{k,h} M_h q_{k,h},
  \qquad k=1,\ldots,N_h,
\end{equation}
where $M_h$ is symmetric positive definite.
Equivalently, $$ B_h = M_h^{-1/2} K_h M_h^{-1/2} $$ is a symmetric matrix with eigenvalues
\{$\lambda_{k,h}\}$, and $$ Z_h(t) = \tr(e^{-t B_h}).
$$
Hence
$$
  \|\mu_h\|_{H'_w}^2 = \int_0^\infty \tr(e^{-t B_h})^2\, w(t)\, dt.
$$

A convenient eigenvalue-free strategy is to approximate the scalar heat kernel by a rational
function on the positive real axis (see \cite{BeylkinMonzon2010}): $$ e^{-t x} \approx \sum_{j=1}^n
  w_j(t)\, \frac{1}{x + \sigma_j(t)}, $$ with positive weights $w_j(t)$ and shifts $\sigma_j(t) > 0$.
Then $$ \tr(e^{-t B_h}) \approx \sum_{j=1}^n w_j(t)\, \tr((B_h + \sigma_j(t) I)^{-1}).
$$
Because
$$
  \tr((B_h + \sigma I)^{-1}) = \tr((K_h + \sigma M_h)^{-1}
  M_h), $$ this computation can be carried out using only shifted linear solves with $K_h + \sigma
  M_h$ and mass-matrix products.
The remaining time integral is approximated by a suitable quadrature rule for the weight $w$.

A practical matrix-based strategy is:
\begin{enumerate}
  \item choose quadrature nodes $\{t_i\}$ and weights $\{\alpha_i\}$ for the outer integral
        \[
          \int_0^\infty Z_h(t)^2 w(t)\,dt;
        \]
  \item for each node $t_i$, approximate the heat trace by a rational expansion of
        $e^{-t_i x}$ and define
        $$
          \widehat Z_h(t_i) = \sum_{j=1}^n w_j(t_i)\, \widehat\tau_j(t_i),
          \qquad
          \widehat\tau_j(t_i) \approx \tr\bigl((K_h + \sigma_j(t_i) M_h)^{-1}
          M_h\bigr); $$ \item estimate each trace term either by direct trace evaluation for small systems or
        by randomized trace estimation for large FEM systems; \item assemble the final approximation $$
          \|\mu_h\|_{H'_w}^2 \approx \sum_i \alpha_i\, [\widehat Z_h(t_i)]^2 w(t_i).
        $$
\end{enumerate}

In practice, multi-shift linear solvers and a small number of random probes can make the method
efficient when $N$ is large.
A reasonable heuristic is that the method is most attractive when the FEM size is too large for
full eigenvalue computation but the number of required quadrature and shift terms remains moderate.
The approximation error depends on the rational approximation, the quadrature rule, and the trace
estimation accuracy; the latter is discussed in Appendix~\ref{app:matrix-dual-norm}.


\subsection{Connection to Weighted Sobolev Spaces}

\subsubsection{Sobolev Characterisation of \texorpdfstring{$H_w$}{the weighted Laplace space}}

For the weight $w(t) = t^{2s-1} e^{-\alpha t}$ (with $s > 0$, $\alpha > 0$), the reproducing kernel
$(1)$ becomes

$$ K(\lambda, \eta) = \frac{\Gamma(2s)}{(\lambda + \eta +
    \alpha)^{2s}}, $$

which is the Green's function of the operator

$$ L = \left(\alpha - \frac{d}{d\lambda}\right)^{2s} $$

on
$(0, \infty)$ — a weighted Sobolev operator of order $2s$ with exponential weight
$e^{-\alpha\lambda}$.
A short RKHS/Green's-function derivation is provided in Appendix~\ref{app:rkhs-greens-function}.

The space $H_w$ is therefore a \emph{weighted Sobolev space of order $s$ on $(0,\infty)$} with
exponential weight, and $H'_w$ is its dual.
This example illustrates the connection between Laplace-transform RKHSs and classical fractional
Sobolev spaces.
Point evaluations are continuous on $H_w$ when $s > 1/2$ (the 1D Sobolev embedding threshold).

\subsubsection{The Inner Product on \texorpdfstring{$H_w$}{the weighted Laplace space}}

For the Laplace kernel, the synthesis operator $S$ in \eqref{eq:weighted-synthesis-transform} is
the Laplace transform $\mathcal{L}$.
On the admissible locally integrable weighted class considered here, uniqueness of the Laplace
transform makes $S$ injective.
Thus \eqref{eq:primal-inverse-transform-inner-product} identifies the unique pre-image of $f\in
  H_w$ as $\mathcal{L}^{-1}f$.
Equivalently, $H_w$ can be described by the inner product on the spectral/Laplace-transform side:

$$ \langle f, g \rangle_{H_w} = \int_0^\infty \hat{f}(t)\, \hat{g}(t)\,
  \frac{dt}{w(t)}, $$

where $\hat{f} = \mathcal{L}^{-1} f$ (the inverse Laplace
transform).
This is a \emph{Sobolev-type norm defined via spectral multipliers}, exactly analogous to the
characterisation $\|u\|_{H^s} \sim \|(1+|\xi|^2)^{s/2} \hat{u}\|_{L^2}$ in the Fourier setting,
with the Laplace transform replacing the Fourier transform.


\subsection{Interpretation and Remarks}

\subsubsection{Physical Interpretation}

The $H'_w$ norm of $\mu$ measures the \emph{$L^2_w$-energy of the heat trace}: the spectral measure
$\mu_h$ of the discrete problem always has *less* heat trace energy than the continuous one,
because the discrete eigenvalues are shifted upward and the system "cools faster" in the discrete
approximation.

\subsubsection{Relation to the Spectral Approximation Error}

Combined with the convergence result $\lambda_{k,h} - \lambda_k = O(h^{2p})$, one obtains a
quantitative version:

$$ \|\mu\|_{H'_w}^2 - \|\mu_h\|_{H'_w}^2 = \int_0^\infty
  \bigl(Z(t)^2 - Z_h(t)^2\bigr)\, w(t)\, dt \geq 0, $$

and asymptotically (using
$Z(t) - Z_h(t) = \sum_k (e^{-\lambda_k t} - e^{-\lambda_{k,h} t}) \leq t \sum_k (\lambda_{k,h} -
  \lambda_k) e^{-\lambda_k t}$):

$$ \|\mu\|_{H'_w}^2 - \|\mu_h\|_{H'_w}^2 =
  O(h^{2p}) \quad \text{as } h \to 0.
$$

\subsubsection{The Inequality Cannot Be Reversed in General}

The inequality $\|\mu_h\|_{H'_w} \leq \|\mu\|_{H'_w}$ is one-directional and relies on the complete
monotonicity of $K$ together with positivity of $w$.
If $w$ changed sign or if a different function space were used, the inequality could fail.

\subsection{Summary}
The key points from this section are:
\begin{enumerate}
  \item The weighted Laplace space $H_w$ is a weighted RKHS of Laplace transforms.
  \item The dual norm of a spectral measure satisfies
        $\|\mu\|_{H'_w}^2 = \displaystyle\int_0^\infty Z(t)^2\, w(t)\, dt$.
  \item For conforming FEM eigenvalues, $Z_h(t) \leq Z(t)$ for all $t > 0$, which implies
        $\|\mu_h\|_{H'_w} \leq \|\mu\|_{H'_w}$.
  \item The finiteness condition $\int_0^\infty Z(t)^2 w(t)\, dt < \infty$ is the precise
        requirement for $\mu \in H'_w$; in particular, for $w(t)=t^\alpha e^{-\delta t}$
        this holds when $\alpha > d-1$.
  \item For every fixed conforming discretization with $N_h<\infty$, one has the
        strict inequality $\|\mu_h\|_{H'_w} < \|\mu\|_{H'_w}$ for any weight $w$
        that is positive on a set of positive measure; nevertheless, the inequality is
        sharp under mesh refinement.
\end{enumerate}

The result \eqref{eq:dual-norm-ineq} holds for \emph{any non-negative weight $w$ such that $\mu,
    \mu_h \in H'_w$}, with no further assumptions on the domain or eigenfunctions beyond FEM eigenvalue
monotonicity.

\bigskip

\nocite{BabuskaOsborn1991,BergForst1975,Aronszajn1950,Widder1941,BergerGauduchonMazet1971,Kac1966}

%% file: rational-approximation.tex
\section{Rational Approximation of Shifted Symbols in \texorpdfstring{$H_w$}{the weighted Laplace space}: Optimal and Classical Constructions}\label{sec:rational-approx}
This section studies rational approximation of the shifted symbol
$\phi_\kappa(x)=(x+\kappa^2)^{-\beta}$ ($\beta>0$, $\kappa>0$) in the weighted Laplace space $H_w$
introduced in Section~\ref{sec:LaplaceSpaces}.
The unshifted symbol $x^{-\beta}$ is considered later as a comparison case in the Stieltjes and
BURA constructions.
Our emphasis is on norm-adapted constructions that directly optimize the $H_w$ error, together with
classical rational approximation families used as benchmarks.
We proceed in two steps: first the optimization-based formulation used in our computations, then
classical constructions used as comparison baselines.
For this paper, we emphasize the computational consequences of the reduced $H_w$ optimization
\begin{table}[ht!]
  \centering\small
  \caption{Comparison with classical spectral and rational approximation bounds.}
  \label{tab:comparison}
  \begin{tabularx}{\textwidth}{|>{\raggedright\arraybackslash}
    X|>{\raggedright\arraybackslash}X|>{\raggedright\arraybackslash}X|} \hline \textbf{Aspect} &
    \textbf{Classical theory}                                                                  & \textbf{This paper}                                                                                                                        \\ \hline FEM spectral control & Eigenvalue
       monotonicity $\lambda_k \leq \lambda_{k,h}$ and $O(h^{2p})$ convergence for individual eigenvalues.
                                                                                               & Dual-norm inequality
       $\|\mu_h\|_{H'_w} \leq \|\mu\|_{H'_w}$ controlling the whole spectral measure
    via heat-trace energy.                                                                                                                                                                                                                  \\
    \hline
    Approximation norm                                                                         & $L^\infty[1,\Lambda]$ or $L^2$ spectral interval norms.
                                                                                               & Weighted Laplace norm $H_w$ adapted to Laplace transforms and
    heat-trace behavior.                                                                                                                                                                                                                    \\
    \hline
    Rational error rate                                                                        & BURA-type bounds $E_n = O(e^{-c n / \log \Lambda})$.
                                                                                               & Conditional $H_w$ error bound, inherited from an external weighted exponential-sum estimate, with spectral-\allowbreak statistics control. \\
    \hline
    Mesh dependence                                                                            & Error bounds often depend explicitly on spectral range
    $\Lambda_h$ and mesh parameters.                                                           & Rational error is mesh-independent once
    $\phi_\kappa \in H_w$, and FEM error enters separately through $\|\mu\|_{H'_w}$.                                                                                                                                                        \\
    \hline
  \end{tabularx}
\end{table}

In Table~\ref{tab:comparison}, ``best'' is always relative to each method's native target metric:
our construction minimizes an $H_w$ objective, whereas classical BURA targets $L^\infty$ on a
finite interval and quadrature-based constructions inherit their own induced metrics.

\subsection{Membership of shifted symbols in \texorpdfstring{$H_w$}{the weighted Laplace space}}

Recall from Section~\ref{sec:LaplaceSpaces} that $H_w$ consists of functions $f : (0,\infty) \to
  \mathbb{R}$ of the form $f = \mathcal{L}g$ for $g \in L^2(w^{-1})$, with norm $\|f\|_{H_w} =
  \|g\|_{L^2(w^{-1})}$.

\begin{theorem}[Joint admissibility of the weight]\label{thm:joint-weight-admissibility}
  Let $\mu$ be the spectral measure from \eqref{eq:spectral-measures}, let
  $\phi_\kappa(x)=(x+\kappa^2)^{-\beta}$ with $\kappa>0$, and let $w(t)=t^\alpha e^{-\delta t}$ with
  $\delta\geq0$.
  Then the simultaneous membership conditions \[ \mu\in H'_w \qquad\text{and}\qquad \phi_\kappa\in
    H_w \] hold if and only if \[ d-1<\alpha<2\beta-1 \qquad\text{and}\qquad 0\leq\delta<2\kappa^2.
  \]
  In particular, the interval of admissible values of $\alpha$ is non-empty
  if and only if $d<2\beta$.
  With the standard Mat\'ern/SPDE parameterization~\cite{Bolin2013,BolinKirchner2020}
  $\beta=\nu+d/2$, this is equivalent to $\nu>0$.
\end{theorem}

\begin{proof}
  By Proposition~\ref{prop:dual-membership-weight}, the condition $\mu\in H'_w$ is equivalent to
  $\alpha>d-1$.

  The shifted symbol has Laplace pre-image
  $g_{\phi_\kappa}(t)=t^{\beta-1}e^{-\kappa^2t}/\Gamma(\beta)$, since

  \[
    \phi_\kappa(x)=\frac{1}{\Gamma(\beta)}\int_0^\infty t^{\beta-1}e^{-(x+\kappa^2)t}\,dt.
  \]
  Therefore its norm is
  \[
    \|\phi_\kappa\|_{H_w}^2
    =\frac{1}{\Gamma(\beta)^2}\int_0^\infty t^{2\beta-2-\alpha}e^{-(2\kappa^2-\delta)t}\,dt.
  \]
  Its behavior near zero requires $2\beta-2-\alpha>-1$, or
  equivalently $\alpha<2\beta-1$.
  Its behavior at infinity requires $2\kappa^2-\delta>0$.
  Thus $\phi_\kappa\in H_w$ is equivalent to $\alpha<2\beta-1$ and $\delta<2\kappa^2$; in that case,
  \[ \|\phi_\kappa\|_{H_w}^2 =\frac{\Gamma(2\beta-1-\alpha)}{\Gamma(\beta)^2
      (2\kappa^2-\delta)^{2\beta-1-\alpha}}.
  \]

  By contrast, the unshifted symbol $x^{-\beta}$ has no exponential decay in its Laplace pre-image.
  If $\delta>0$, its corresponding weighted norm contains the factor $e^{\delta t}$ and therefore
  diverges at infinity.
  If $\delta=0$, integrability near zero requires $\alpha<2\beta-1$, while integrability at infinity
  requires $\alpha>2\beta-1$; these conditions are incompatible.
  Thus the positive shift is essential in this setting.

  Combining these conditions gives the stated result.
  The interval $(d-1,2\beta-1)$ is non-empty exactly when $d<2\beta$, and substituting
  $\beta=\nu+d/2$ gives $d<2\nu+d$, equivalently $\nu>0$.
\end{proof}

\textbf{Remark on the Mat\'ern parameterization.}
With the standard Mat\'ern/SPDE parameterization \cite{Bolin2013,BolinKirchner2020} \[
  \beta=\nu+\frac d2, \] the joint admissibility condition becomes \[ d-1<\alpha<2\nu+d-1
  \qquad\text{or equivalently}\qquad 0<\alpha-(d-1)<2\nu.
\]
The interval is non-empty precisely when $\nu>0$, equivalently $d<2\beta$.
Thus the two-sided condition on $\alpha$ is compatible with the usual Mat\'ern condition used in
the Bolin--Kirchner SPDE framework.

\subsection{Rational Approximation and Exponential Sums}

\subsubsection{The Correspondence}

A rational function with simple real poles in $(-\infty,-\kappa^2)$ has the partial-fraction form

\[ r_n(x) = \sum_{j=1}^n \frac{c_j}{x + \kappa^2 + \gamma_j}, \qquad c_j \in
  \mathbb{R},\qquad \gamma_j > 0.
\]

In this subsection, this is an admissible-class restriction (not a generic representation of all
real rational functions): we optimize over simple real poles shifted by $-\kappa^2$ so that, after
Laplace inversion, the ansatz is a real sum of decaying exponentials.
Complex-conjugate pole pairs are therefore excluded here, since they produce damped oscillatory
terms $e^{-\alpha t}(a\cos(\beta t)+b\sin(\beta t))$ rather than the exponential-sum dictionary
used below.

Its Laplace pre-image (the function $g_{r_n}$ such that $r_n = \mathcal{L}g_{r_n}$) is \[
  g_{r_n}(t) = \sum_{j=1}^n c_j e^{-(\kappa^2+\gamma_j)t}.
\]

The $H_w$ error of rational approximation is therefore equivalent to
approximating $t^{\beta-1}e^{-\kappa^2t}/\Gamma(\beta)$ by an $n$-term exponential sum
in the weighted $L^2(w^{-1})$ norm:

\begin{equation}\label{eq:hw-rational-error}
  \|\phi_\kappa - r_n\|_{H_w}^2
  = \int_0^\infty \left|
  \frac{t^{\beta-1}e^{-\kappa^2t}}{\Gamma(\beta)} -
  \sum_{j=1}^n c_j e^{-(\kappa^2+\gamma_j)t}
  \right|^2 \frac{dt}{w(t)}.
\end{equation}

\subsection{Optimal \texorpdfstring{$H_w$}{weighted-Laplace} Approximation by Variable Projection}

To align construction with the metric used in Section~\ref{sec:num-verification}, we optimize poles
and coefficients by direct minimization of the $H_w$ objective.

For cleaner formulas, define \[ f_{\beta,\kappa}(t)=t^{\beta-1}e^{-\kappa^2t}, \qquad
  g_n(t;\theta)=\sum_{j=1}^n c_j e^{-(\kappa^2+\gamma_j)t}, \qquad \theta=(c,\gamma),\;\gamma_j>0, \]
and \[ \mathcal{J}_n(c,\gamma):=\int_0^\infty \left|f_{\beta,\kappa}(t)-g_n(t;\theta)\right|^2
  \frac{dt}{w(t)}.
\]
This is equivalent to minimizing $\|\phi_\kappa-r_n\|_{H_w}^2$ up to the
constant factor $\Gamma(\beta)^{-2}$.

Introduce the weighted inner product and norm on $L^2(w^{-1})$ by \[ \langle
  u,v\rangle_{w^{-1}}:=\int_0^\infty u(t)v(t)\,\frac{dt}{w(t)}, \qquad \|u\|_{w^{-1}}^2:=\langle
  u,u\rangle_{w^{-1}}.
\]
Also set $e_j^\gamma(t):=e^{-(\kappa^2+\gamma_j)t}$.

\begin{proposition}[Quadratic form and elimination of linear coefficients]
  \label{prop:quadratic-form-elimination}
  For fixed poles $\gamma$, the objective is \[ \mathcal{J}_n(c,\gamma)=a-2\,b(\gamma)^T c + c^T
    A(\gamma)c, \] with \[ a=\|f_{\beta,\kappa}\|_{w^{-1}}^2, \qquad b_j(\gamma)=\langle
    f_{\beta,\kappa},e_j^\gamma\rangle_{w^{-1}}, \qquad A_{jk}(\gamma)=\langle
    e_j^\gamma,e_k^\gamma\rangle_{w^{-1}}.
  \]
  If $A(\gamma)$ is invertible, then
  \[
    c^*(\gamma)=A(\gamma)^{-1}b(\gamma),
    \qquad
    \widehat{\mathcal{J}}_n(\gamma):=
    \min_c\mathcal{J}_n(c,\gamma)=a-b(\gamma)^T A(\gamma)^{-1}b(\gamma).
  \]
\end{proposition}
For completeness, the elimination identity gives the following best-error characterization.
\begin{proposition}[Best $H_w$-error identity]
  \label{prop:best-hw-error}
  The squared best $n$-term rational approximation error in $H_w$ is \[
    E_n:=\inf_\gamma\|\phi_\kappa-r_n\|_{H_w}^2
    =\frac{1}{\Gamma(\beta)^2}\inf_\gamma\widehat{\mathcal{J}}_n(\gamma)
    =\frac{1}{\Gamma(\beta)^2}\inf_{c\in\mathbb{R}^n,\;\gamma}\mathcal{J}_n(c,\gamma).
  \]
\end{proposition}

The closed-form formulas and gradient in Proposition~\ref{prop:model-weight-closed-forms} are
central for the application: they turn the $H_w$-adapted rational approximation into an explicit
computational pipeline for pole optimization.
As a result, each optimization step can be evaluated without numerical quadrature, improving both
robustness and efficiency in practice.
The resulting approximants and error trends are tested in Section~\ref{sec:num-verification}.

\subsection{Convergence as the Number of Poles Increases}

The rate-transfer proposition above gives the organizing principle for convergence as the number of
poles increases.
It does not by itself establish a rate for the best approximation error $E_n$; any such rate must
come from an external approximation result for the admissible shifted exponential-sum class.

\subsubsection{Best Uniform Rational Approximation (BURA)}

For FEM applications, Hofreither \cite{Hofreither2021} introduced BURA.
It is the best rational approximant in $L^\infty$ norm on the interval $[1,\Lambda_h]$.
Here $\Lambda_h$ is the spectral condition number, defined by
\begin{equation*}
  \Lambda_h = \frac{\lambda_{N,h}}{\lambda_{1,h}}.
\end{equation*}

Choose a best approximant \[ r_n^* \in \operatorname*{arg\,min}_{r \in \mathcal{R}_{n,n}} \max_{x
    \in [1,\Lambda_h]} |x^\beta r(x) - 1|.
\]
Thus $r_n^*$ is a rational function, not the scalar minimum value.

The BURA interval error is commonly written as

\[ \mathcal
  E_n^{\mathrm{BURA}}(\beta, \Lambda_h) := \max_{x \in [1,\Lambda_h]} |x^\beta r_n^*(x) - 1|.
\]
Classical results may provide bounds of the form
\[
  \mathcal E_n^{\mathrm{BURA}}(\beta, \Lambda_h) \leq C \exp\!
  \left(-\frac{\pi^2 n}{\log \Lambda_h}\right).
\]
In practice we use these constructions as external baselines to evaluate the
benefit of $H_w$-adapted optimization; implementation details follow
standard Remez/Zolotarev workflows \cite{Hofreither2021,Stahl1997}.

\subsubsection{Plausible asymptotics and conditional rate transfer}

The BURA discussion also clarifies why a separate global approximation result is needed.
BURA controls a pointwise relative error on a finite spectral interval, whereas the $H_w$ norm is
defined through Laplace pre-images on $(0,\infty)$.
Consequently, a BURA interval estimate does not by itself imply an $H_w$ estimate; a conversion
would require a separate stability theorem connecting these two norms.

The exact identity \eqref{eq:hw-rational-error} shows that the relevant approximation problem is
the approximation of \[ f_{\beta,\kappa}(t)=t^{\beta-1}e^{-\kappa^2t} \] by admissible exponential
sums in the weighted norm $L^2(w^{-1})$.
Classical constructions suggest rapid convergence as the number of poles increases, but the exact
rate depends on the admissible pole class and on how the endpoints $t=0$ and $t=\infty$ are
treated.
Near $t=0$ the target has algebraic behavior $t^{\beta-1}$, while its tail has exponential decay
governed by $\kappa^2$.
Thus finite-interval rational or exponential-sum estimates do not by themselves establish a global
$H_w$ estimate.

We, the authors, believe that under suitable conditions on the parameters and the admissible pole
class, a result of the following type should hold: there exist constants $C>0$ and $\rho>0$ such
that \[ \inf_{c,\gamma}\left\|f_{\beta,\kappa} -\sum_{j=1}^n c_j
  e^{-(\kappa^2+\gamma_j)\,\cdot}\right\|_{L^2(w^{-1})} \leq C e^{-\rho n}.
\]
We do not prove this estimate here.
If such an estimate were established, then the normalization in the definition of $E_n$ would give
the conditional bound \[ E_n \leq \Gamma(\beta)^{-2}C^2 e^{-2\rho n}.
\]
The constants and the range of parameters for which this rate holds remain to be determined.

\subsubsection{Stieltjes Integral as a Possible Proof Route}

For $0<\beta<1$, a possible route to an external estimate is the Stieltjes representation

\[ x^{-\beta} = \frac{\sin(\pi\beta)}{\pi} \int_0^\infty \frac{t^{1-\beta}}{x +
    t}\, \frac{dt}{t}, \] which is a weighted integral of $1/(x+t)$ against the measure $\mu_\beta =
  \frac{\sin(\pi\beta)}{\pi} t^{-\beta} dt$.
Applying an $n$-point Gauss quadrature rule with nodes $\{t_j\}$ and weights $\{w_j^G\}$ for the
measure $\mu_\beta$ gives:

\[ r_n^G(x) = \sum_{j=1}^n \frac{w_j^G}{x + t_j}
  \approx x^{-\beta}.
\]
This is an $(n-1,n)$ rational function with positive weights and positive poles, a structure that
is natural for $r_n \in H_w$ and for operator-kernel applications.

For shifted symbols, the same quadrature idea is applied after replacing $x$ by $x+\kappa^2$,
producing the admissible shifted poles used above.
To turn this construction into the global weighted estimate, one would still need a quadrature
error analysis together with control of the $t=0$ and $t=\infty$ tails in the weighted Laplace
norm.
The representation is also directly applicable only for $0<\beta<1$; larger values require a
generalized Stieltjes representation or another external approximation theorem.

Thus the Stieltjes and BURA constructions provide possible external convergence inputs or
comparison baselines, while the variable-projection construction is optimized directly in the $H_w$
metric.
The numerical experiments in Section~\ref{sec:num-verification} compare these approaches at finite
$n$.

%% file: matern-application.tex
\section{Application to Gaussian Matérn Fields}

\subsection{Matérn Fields via the SPDE Representation}

A Gaussian Matérn field $u$ on $\Omega \subset \mathbb{R}^d$ with smoothness $\nu > 0$ and
correlation length $\kappa^{-1}$ satisfies the SPDE
\cite{LindgrenRueLindstrom2011,Bolin2013,BolinKirchner2020}

\[ (\kappa^2 -
    \Delta)^{(\nu + d/2)/2} u = \mathcal{W}, \]

where $\mathcal{W}$ is spatial
white noise.
Setting $\beta = \nu + d/2$, the covariance operator is $\mathcal{C} = (\kappa^2 -
    \Delta)^{-\beta}$, with spectral representation

\[ \mathcal{C} u_k = (\lambda_k
    + \kappa^2)^{-\beta} u_k, \]

where $\{\lambda_k, u_k\}$ are
Dirichlet--Laplacian eigenpairs on $\Omega$.
The shifted spectral symbol is $\phi_\kappa(x) = (x + \kappa^2)^{-\beta}$.

\subsection{Rational Approximation → Mixture of Simple SPDEs}

A rational approximant to the shifted spectral symbol

\[ r_n(x) = \sum_{j=1}^n
    \frac{c_j}{x + \kappa^2 + \gamma_j} \]

replaces the covariance operator by

\[ \mathcal{C}_{r_n} = \sum_{j=1}^n c_j\, (\kappa^2 + \gamma_j - \Delta)^{-1}.
\]

Each summand $(\kappa_j^2 - \Delta)^{-1}$ (with $\kappa_j^2 = \kappa^2 + \gamma_j$) is a resolvent
covariance operator with modified length scale $\kappa_j^{-1}$.
In the standard SPDE parametrization, it corresponds to a Matérn field with covariance exponent
$1$, or smoothness $\nu=1-d/2$; in one spatial dimension this is the exponential Matérn-$1/2$
field.
Therefore:

\begin{quote}
    The rational approximation of the Matérn covariance symbol gives a finite linear combination of
    resolvent covariance operators, each solvable as a standard elliptic PDE.
    If the coefficients $c_j$ are nonnegative, this combination is the covariance operator of a sum of
    independent resolvent fields.
\end{quote}

This is the approach of \cite{Bolin2013,BolinKirchner2020}, and gives a computationally tractable
approximation with sparse precision matrices.

\subsection{Error in Spectral Statistics via the Duality Bound}

Using the spectral measures defined in \eqref{eq:spectral-measures}, the covariance operators have
the spectral trace representations
\begin{equation*}
    \mathbb{E}(\|u\|^2_{L^2(\omega)})=\operatorname{tr}(\mathcal{C}) = \sum_{k=1}^{\infty} \phi_\kappa(\lambda_k),
\end{equation*}
\begin{equation*}
    \mathbb{E}(\|u_{r_n,h}\|^2_{L^2(\omega)})=\operatorname{tr}(\mathcal{C}_{r_n,h}) = \sum_{k=1}^{N_h} r_n(\lambda_{k,h}),
\end{equation*}
where $\mathcal{C}_{r_n,h}=r_n(A_h)$ and $A_h$ is the discrete Dirichlet Laplacian.
Both expressions give the expected energy (squared $L^2$ norm) of the solution.
Thus the error in the trace of the approximate covariance operator decomposes as:

\[ \operatorname{tr}(\mathcal{C}) - \operatorname{tr}(\mathcal{C}_{r_n,h}) =
    \underbrace{\langle \mu - \mu_h,\, \phi_\kappa \rangle}_{\text{FEM error}} + \underbrace{\langle
        \mu_h,\, \phi_\kappa - r_n \rangle}_{\text{rational approx.
                error}}.
\]

\textbf{FEM error term.}
By the dual-norm inequality \eqref{eq:dual-norm-ineq} proved in Section~\ref{sec:LaplaceSpaces}:

\[ |\langle \mu - \mu_h, \phi_\kappa \rangle| \leq \|\mu - \mu_h\|_{H'_w}\,
    \|\phi_\kappa\|_{H_w} \leq C\, h^{2p}\, \|\phi_\kappa\|_{H_w}.
\]

\textbf{Rational approximation error term.}
By Cauchy--Schwarz and $\|\mu_h\|_{H'_w} \leq \|\mu\|_{H'_w}$:

\[ |\langle
    \mu_h, \phi_\kappa - r_n \rangle| \leq \|\mu_h\|_{H'_w}\, \|\phi_\kappa - r_n\|_{H_w} \leq
    \|\mu\|_{H'_w}\, \|\phi_\kappa - r_n\|_{H_w}.
\]

\textbf{Combined bound.}
Setting $E_n = \|\phi_\kappa - r_n\|_{H_w}$:

\begin{equation}\label{eq:matern-trace-error-bound}
    |\operatorname{tr}(\mathcal{C}) -
    \operatorname{tr}(\mathcal{C}_{r_n,h})| \leq \|\phi_\kappa\|_{H_w}\, C h^{2p} + \|\mu\|_{H'_w}\, E_n.
\end{equation}

To balance the two contributions, one should choose $n$ so that $E_n \lesssim h^{2p}$.
A more specific relation between $n$ and $h$ requires an actual convergence theorem for $E_n$; the
rational-approximation section only gives a conditional rate-transfer statement and does not prove
such a rate here.

\subsection{Summary of the Full Framework}

The chain of ideas can be summarized as:

\[
    \begin{gathered}
        \text{Matérn covariance} \\
        \Downarrow \\
        \phi_\kappa(\lambda) = (\lambda + \kappa^2)^{-\beta} \in H_w \\
        \Downarrow \\
        \text{rational approximation in }
        H_w \\ \Downarrow \\ \text{Mixture of } n \text{ resolvent covariance fields} \\ \Downarrow \\
        |\operatorname{tr}(\mathcal{C}) - \operatorname{tr}(\mathcal{C}_{r_n,h})| \leq
        \|\phi_\kappa\|_{H_w} \cdot C h^{2p} + \|\mu\|_{H'_w} \cdot E_n.
    \end{gathered}
\]

\subsection*{Key Parameter Choices}

\begin{center}
    \begin{tabularx}{\textwidth}{>{\raggedright\arraybackslash}l >{\raggedright\arraybackslash}l X}
        \textbf{Parameter}         & \textbf{Admissible range}       & \textbf{Role}                                               \\
        \hline
        $\alpha$ (weight exponent) & $d-1 < \alpha < 2\nu + d - 1$   & Ensures both $\mu \in H'_w$ and $\phi_\kappa \in H_w$.      \\
        $\delta$ (weight decay)    & $0 \leq \delta < 2\kappa^2$     & Ensures decay of the Laplace pre-image of $\phi_\kappa$.    \\
        $n$ (number of poles)      & chosen so $E_n \lesssim h^{2p}$ & Balances rational and FEM errors.                           \\
        $\beta = \nu + d/2$        & $\nu > 0$                       & Mat\'ern smoothness; larger $\beta$ improves approximation. \\
    \end{tabularx}
\end{center}

%% file: numerical-verification.tex
\section{Numerical Verification}
\label{sec:num-verification}
This section presents numerical experiments that illustrate the main qualitative features of the
theory developed in this paper.
We verify first the conforming finite element inequality \(\|\mu_h\|_{H'_w} \leq \|\mu\|_{H'_w}\)
and then the practical accuracy gains of metric-adapted rational approximation in the primal space
$H_w$.
The approximation families are exactly those introduced in Section~\ref{sec:rational-approx}: an
optimization-based variable-projection construction and baseline comparison methods.
The experiments are performed on the unit square $\Omega = (0,1)^2$ with homogeneous Dirichlet
boundary conditions.
In addition to the primary triangular-mesh computations, we use a tensor-product FEM construction
on aligned Cartesian meshes for efficient spectral benchmarks.

\subsection{Computational Setup}
We discretize the Laplace eigenvalue problem on $\Omega=(0,1)^2$ with standard piecewise linear
conforming finite elements on uniform triangular meshes $h=2^{-k}$ for $k=3,4,5,6$.
For the dual norm experiment we fix $w(t)=t^\alpha e^{-\gamma t}$ with $(\alpha,\gamma)=(2,1)$ and
compute the first $N=200$ discrete eigenvalues on each mesh.
The dual norm is evaluated through the heat-trace representation \[ \|\mu_h\|_{H'_w}^2 =
  \int_0^\infty Z_h(t)^2\,w(t)\,dt, \qquad Z_h(t)=\sum_{k=1}^N e^{-\lambda_{k,h}t}, \] using a
logarithmic quadrature rule in $t$.
The same quadrature is used for the continuous reference norm.

For the rational approximation experiment we consider the shifted Mat\'ern symbol \[
  \phi(x)=(x+\kappa^2)^{-\beta}, \qquad \kappa=1, \qquad \beta=2.5, \] and compare several $n$-pole
rational approximants in the primal norm $H_w$.
To illustrate the dependence on the metric, we report two representative weights,
$(\alpha,\gamma)=(0.25,1)$ and $(0.5,1)$.
For aligned Cartesian benchmark meshes we also exploit the tensor-product structure of the
spectrum, so that auxiliary comparisons can be assembled efficiently from 1D eigenvalues.
The full tensor-product derivation and the associated 2D/3D formulas are collected in
Appendix~\ref{app:tensor-product-fem}.
As an additional consistency check, we verified the dual norm by a matrix-based shifted-solve
strategy; the algorithmic details are deferred to Appendix~\ref{app:matrix-dual-norm}.

For reproducibility, all approximation methods are compared at the same rational degree $n$ and
evaluated with the same quadrature-based $H_w$ error pipeline.
The optimization-based method uses gradient-based variable projection with analytic reduced
gradients and box constraints on log-poles (L-BFGS-B), initialized by one deterministic log-grid
start plus randomized restarts (default four starts in total, fixed seed).
The fixed-grid VP and joint NLS baselines use the same target symbol, weights, and degree range.

\subsection{Verification of the Dual Norm Inequality}
The key qualitative prediction of Theorem~\ref{thm:main-dual-norm} is that the conforming discrete
spectral measure has no larger $H'_w$ norm than the continuous spectral measure.
In the unit-square experiment we approximate both norms and monitor their difference as the mesh is
refined.

The continuous norm is approximated using the reference eigenvalues of the continuous Laplacian on
the unit square, $\lambda_{mn} = \pi^2(m^2+n^2)$ for $m,n\in\mathbb{N}$, truncated at the same
spectral index.

In the reported computations we observe:
\begin{itemize}
  \item the inequality
        $\|\mu_h\|_{H'_w} \leq \|\mu\|_{H'_w}$ for every conforming mesh;
  \item a decreasing gap
        $\|\mu\|_{H'_w}-\|\mu_h\|_{H'_w}$ as $h\to 0$, reflecting the convergence
        of the discrete heat trace to the continuous heat trace;
  \item numerical evidence that the dual norm is stable under mesh refinement.
\end{itemize}

\begin{figure}[ht!]
  \centering
  \includegraphics[width=0.80\linewidth]{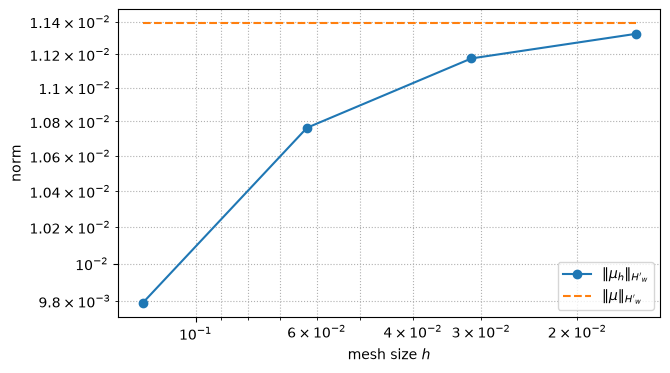}
  \caption{Computed values of $\|\mu_h\|_{H'_w}$ and the continuous reference
    $\|\mu\|_{H'_w}$ versus mesh refinement on the unit square, for
    $w(t)=t^2e^{-t}$ and spectral truncation $N=200$.
    The discrete curve remains below the continuous reference on all tested meshes, consistent with
    $\|\mu_h\|_{H'_w}\leq \|\mu\|_{H'_w}$.
  }
  \label{fig:dual-norm-inequality}
\end{figure}

\begin{figure}[ht!]
  \centering
  \includegraphics[width=0.80\linewidth]{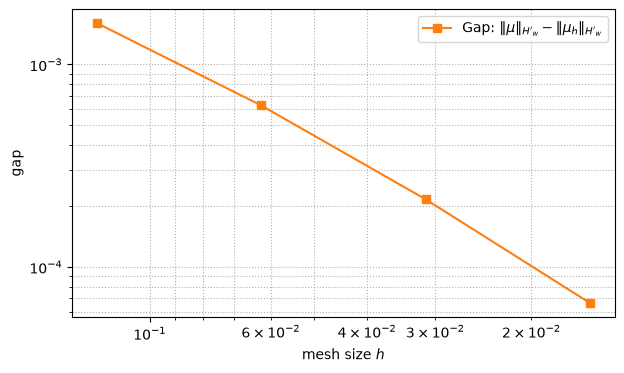}
  \caption{Gap $\|\mu\|_{H'_w} - \|\mu_h\|_{H'_w}$ versus mesh size for the
    unit-square verification problem.}
  \label{fig:dual-norm-gap-2d}
\end{figure}

The computations confirm the expected monotonicity pattern: the inequality $\|\mu_h\|_{H'_w} \leq
  \|\mu\|_{H'_w}$ is preserved over the tested meshes, and Figure~\ref{fig:dual-norm-gap-2d} shows a
decreasing gap under refinement.
Together with Figure~\ref{fig:dual-norm-inequality}, this gives direct numerical support for the
qualitative prediction of Theorem~\ref{thm:main-dual-norm} in the unit-square setting.
Additional numerical checks with moderately varied spectral truncation levels and log-quadrature
resolutions (not shown) produced only small changes in the reported values and did not alter the
observed inequality or gap trend.

\subsection{Rational Approximation Experiments}
To illustrate the approximation theory from Section~\ref{sec:rational-approx}, we compare three
families of $n$-pole rational approximants for the shifted symbol $\phi(x)=(x+1)^{-2.5}$.
The first family minimizes the reduced $H_w$ objective directly by gradient-based variable
projection over the pole locations.
The second uses a fixed logarithmic pole grid with optimal linear coefficients.
The third is a joint nonlinear least-squares fit in Laplace space, included as a flexible baseline
that does not explicitly optimize $H_w$.

For each weight we plot the relative error \[ \frac{\|\phi-r_n\|_{H_w}}{\|\phi\|_{H_w}} \] for
$n=0,1,\dots,10$.
This comparison focuses on the most practical question: does optimizing the poles for the $H_w$
metric produce a visible gain over more generic constructions at modest rational degree?

\begin{figure}[ht!]
  \centering
  \includegraphics[width=0.92\linewidth]{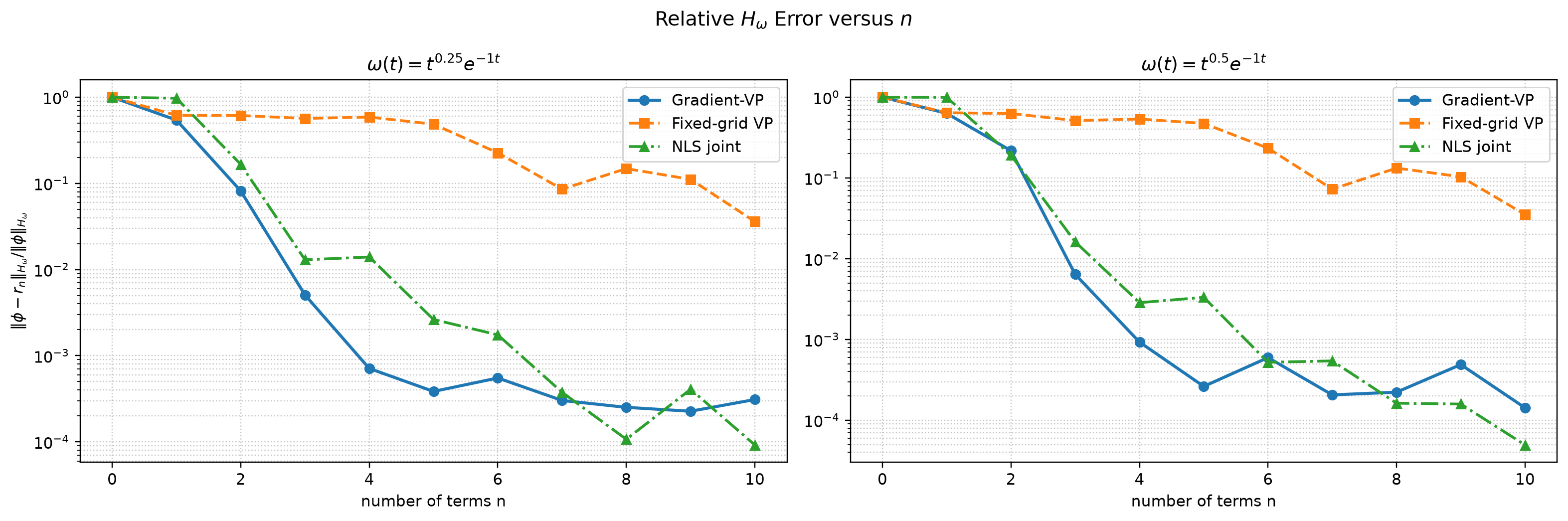}
  \caption{Relative $H_w$ error $\|\phi-r_n\|_{H_w}/\|\phi\|_{H_w}$ versus
    number of poles $n$ for $\phi(x)=(x+1)^{-2.5}$ and two representative
    weights.
    Methods: gradient-VP (target-metric optimization), fixed-grid VP, and joint NLS in Laplace space.
    At the same $n$, gradient-VP is typically the most accurate or competitive across the tested range.
  }
  \label{fig:homega-relative-error}
\end{figure}

Figure~\ref{fig:homega-relative-error} shows that the metric-adapted optimization substantially
improves the primal error over the fixed-grid construction and remains competitive with the more
flexible nonlinear least-squares baseline across the tested range of $n$.
The decay is rapid and qualitatively consistent with the rapid-convergence picture from
Section~\ref{sec:rational-approx}, but the main point of the experiment is comparative rather than
asymptotic: the $H_w$-optimized approximant is already more accurate at small and moderate values
of $n$.

\subsection{Application-facing Spectral-sum Test}
To complement the norm-level comparison, we evaluate each rational approximant on a discrete
spectral observable, \[ S_h(f)=\sum_k f(\lambda_{k,h}), \] and report the relative error \[
  E^{\mathrm{spec}}_{n,h} = \frac{|S_h(\phi)-S_h(r_n)|}{|S_h(\phi)|}.
\]
The sums are evaluated on aligned Cartesian $Q_1$ meshes using the
tensor-product spectrum construction from Appendix~\ref{app:tensor-product-fem}.

\begin{figure}[ht!]
  \centering
  \includegraphics[width=0.92\linewidth]{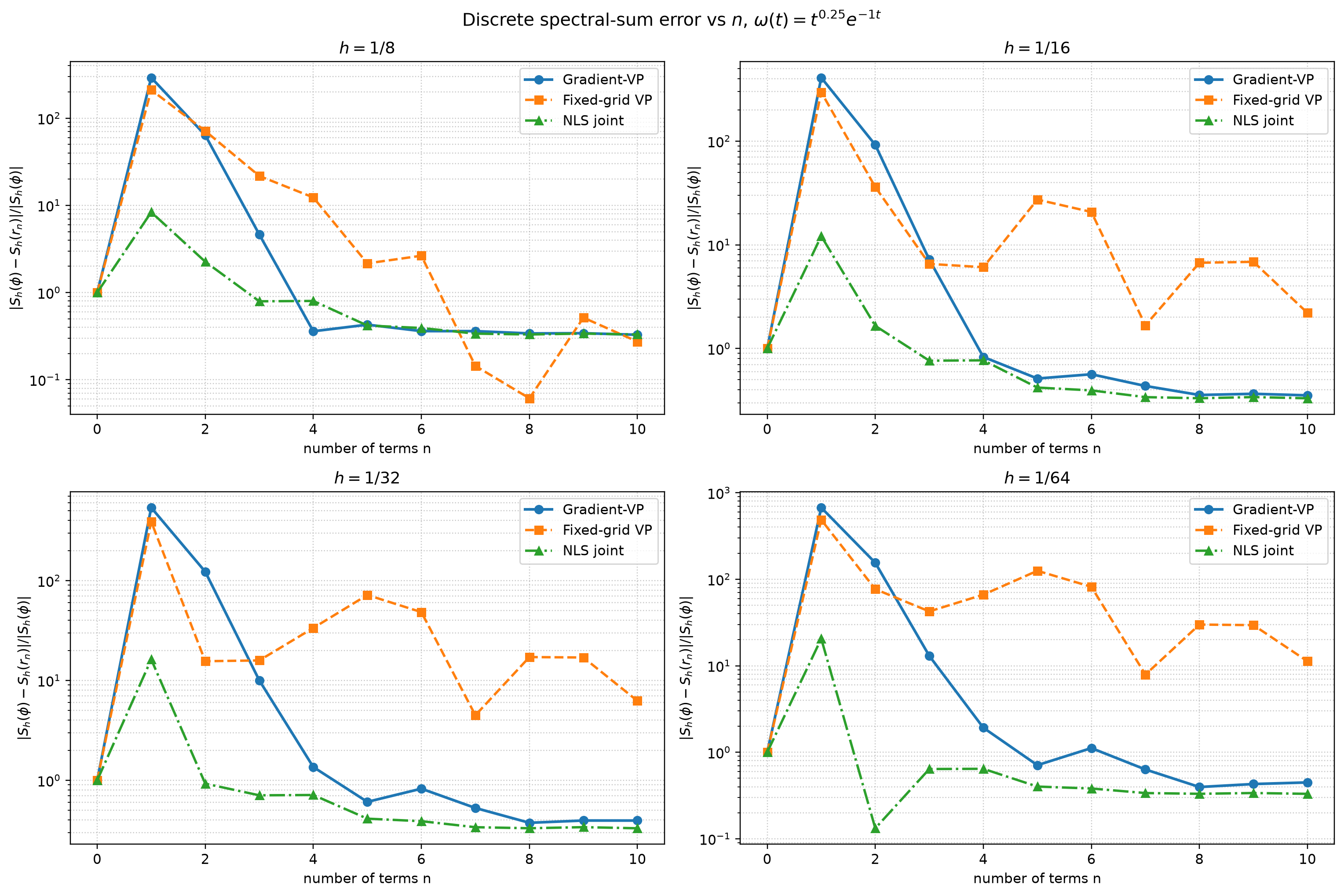}
  \caption{Relative spectral-sum error $E^{\mathrm{spec}}_{n,h}$ versus number
    of poles $n$, shown for several mesh sizes on the unit square.}
  \label{fig:homega-spectral-sum-vs-n}
\end{figure}

\begin{figure}[ht!]
  \centering
  \includegraphics[width=0.92\linewidth]{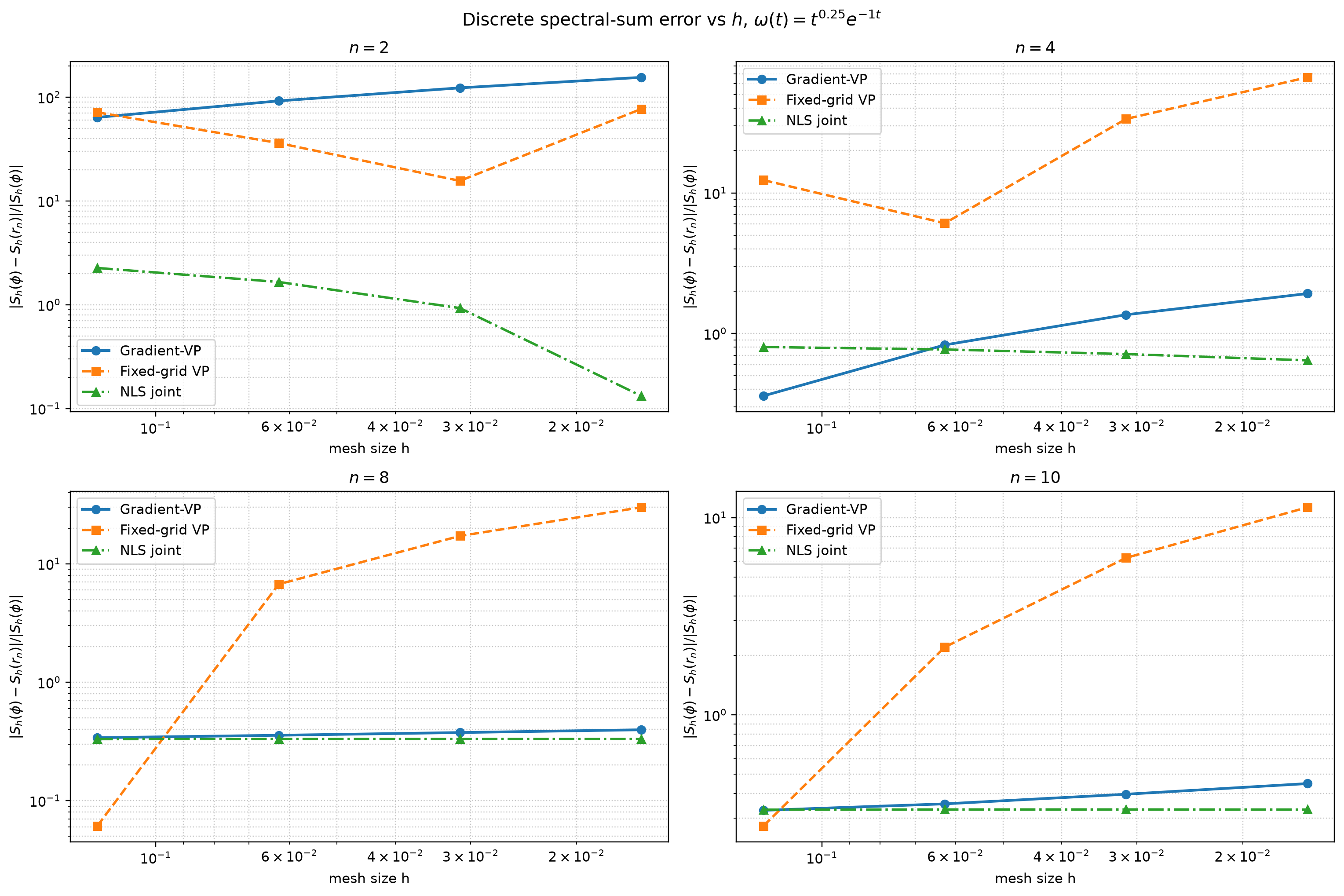}
  \caption{Relative spectral-sum error $E^{\mathrm{spec}}_{n,h}$ versus mesh
    size $h$ for selected values of $n$.}
  \label{fig:homega-spectral-sum-vs-h}
\end{figure}

Figure~\ref{fig:homega-spectral-sum-vs-n} shows that increasing $n$ reduces the observable-level
error across all tested meshes, while Figure~\ref{fig:homega-spectral-sum-vs-h} shows that the same
approximation families remain effective under refinement.
This test links the $H_w$-optimization directly to a concrete spectral quantity and complements the
norm-level comparison in Figure~\ref{fig:homega-relative-error}.

\subsection{Discussion}
The experiments in this section are intended to illustrate two complementary
aspects of the theory:
\begin{itemize}
  \item the discrete spectral measure of conforming FEM does not exceed the
        continuous spectral measure in the weighted Laplace dual norm;
  \item rational approximation benefits from optimization in the target $H_w$
        metric, leading to visibly smaller errors than more
        generic constructions at comparable rational degree;
  \item the same metric-adapted approximants improve an application-facing
        discrete spectral observable across both rational degree and mesh refinement.
\end{itemize}

Taken together, the dual norm plots, the $H_w$-error comparison, and the spectral-sum test give a
concise numerical narrative: the FEM-side monotonicity predicted by the theory is visible under
mesh refinement, and the primal-side optimization is useful not only asymptotically but already in
the low- to moderate-order regime most relevant for computation and observable-level accuracy.

%% file: conclusions.tex
\section{Conclusions and Outlook}
\label{sec:conclusions}

This paper has introduced weighted transform spaces as an abstract framework and weighted Laplace
spaces $H_w$ as its concrete specialization for comparing continuous and discrete spectral measures
of elliptic operators.
The main theoretical contribution is the dual-norm inequality

\[
    \|\mu_h\|_{H'_w} \leq \|\mu\|_{H'_w}, \]

which follows directly from
monotonicity of conforming FEM eigenvalues and the heat-trace representation of the dual norm.
The inequality is both simple and robust: it holds for any conforming finite element discretization
and any non-negative weight $w$ for which the spectral measures belong to $H'_w$.
Via dual pairing with test functions in $H_w$, this norm control transfers directly to uniform
bounds for finite spectral sums and related transformed observables.

A second contribution is the $H_w$-adapted rational approximation of shifted symbols
$\phi_\kappa(x)=(x+\kappa^2)^{-\beta}$ through explicit variable-projection optimization.
The weighted Laplace framework also provides a conditional transfer principle: any verified
estimate for the corresponding weighted Laplace pre-image approximation yields an $H_w$ error
estimate.
This makes the framework relevant to potential applications such as Gaussian covariance
approximation \cite{LangLarssonSchwab2013}, where rational approximants may represent finite
combinations of simpler covariance operators; developing such applications is beyond the scope of
this paper.

The numerical verification section provides computational evidence for these conclusions in the
unit-square experiments.
The theoretical inequality is independent of these experiments; numerically, the computations
preserve $\|\mu_h\|_{H'_w} \leq \|\mu\|_{H'_w}$ under mesh refinement and show a decreasing
dual-norm gap.
The rational-approximation experiments further show that optimization in the target $H_w$ metric
yields smaller errors than the baseline constructions over the tested weights and rational degrees.

Together, the analytical and numerical results establish a coherent picture in which conforming FEM
discretization error and rational approximation error are quantified in the same functional
framework.

Before outlining extensions, we stress the present scope.
The analysis is carried out for the Dirichlet Laplacian, the central inequality relies on
conforming FEM eigenvalue monotonicity, and the rational ansatz is restricted to the structured
class used in our Laplace-domain formulation.
No universal convergence rate for the rational approximation problem is asserted here; rate
statements are conditional on corresponding weighted exponential-sum estimates.
The directions below are therefore plausible extensions rather than results proved here.

\subsection{Extensions to Other Operators and Spectral Measures}

Although our analysis is phrased for the Dirichlet Laplacian, the underlying ideas are
substantially more general.
The key ingredients are:

\begin{itemize}
    \item a discrete spectrum with conforming or monotone discretization
          properties;
    \item a spectral measure that admits a meaningful dual $H'_w$ norm via heat
          trace or a related kernel representation; and
    \item a class of symbols $\phi$ whose Laplace pre-images belong to the
          primal space $H_w$.
\end{itemize}

These ingredients also appear for a wide class of second-order elliptic operators with smooth
coefficients, for Schr\"odinger operators on bounded domains, and more generally for self-adjoint
positive operators with compact resolvent.
For such extensions, the dual-norm approach is expected to remain valid only after verifying a
suitable monotone discretization, a heat-trace or related kernel representation, and membership of
the relevant spectral measures in $H'_w$.

The same perspective is promising for non-Laplacian spectral measures, including fractional
elliptic operators, variable-coefficient operators, and spectral models on graphs or networks.
The relevant transform or kernel representation may need to be adapted in each case: the dual norm
is driven by the resulting reproducing kernel $K(\lambda,\eta)$, which need not be the same Laplace
kernel induced by $w$ in the present setting.

Thus, the following are proposed directions for future investigation rather than results established
in this paper:

\begin{itemize}
    \item extensions to elliptic operators with non-constant coefficients and
          mixed boundary conditions;
    \item spectral measures of fractional Laplacians and other non-local operators
          on bounded domains;
    \item analogous dual-norm inequalities for graph Laplacians and discrete
          operators arising in network models;
    \item alternative weight choices $w$ that are adapted to different spectral
          growth regimes or to operators with singular potentials.
\end{itemize}

By making these extensions, the weighted Laplace perspective may become a useful tool for a broader
class of spectral approximation problems, including those outside the classical setting of the
Dirichlet Laplace operator.

\subsection{Final Remarks}

The weighted Laplace framework provides a unified way to think about both spectral measure
comparison and rational approximation.
Its strength lies in combining the classical heat-trace viewpoint with modern RKHS techniques, and
it provides a concrete $H_w$-adapted optimization framework whose numerical behavior can be tested
at both norm and observable levels.
This combination offers a path for future work on spectral approximation in high-dimensional and
non-standard settings.

%% file: appendix-rkhs-greens-functions.tex
\section{RKHS and Green's-Function Interpretation}\label{app:rkhs-greens-function}

This appendix summarizes the reproducing kernel and Green's-function
perspective that underlies the weighted Sobolev interpretation of $H_w$.

\subsection{Reproducing Kernel Hilbert Spaces}
A Hilbert space $H$ of real-valued functions on a domain $\Omega$ is a
reproducing kernel Hilbert space (RKHS) if point evaluation is a bounded
linear functional on $H$. For each $y \in \Omega$, there exists a unique
function $K(\cdot,y) \in H$ such that

$$
u(y) = \langle u, K(\cdot,y) \rangle_H \quad \text{for all } u \in H.
$$

The function $K : \Omega \times \Omega \to \mathbb{R}$ is the reproducing
kernel of $H$. It is symmetric and positive definite.

\subsection{When the Kernel Is a Green's Function}
Consider an elliptic, self-adjoint, positive-definite operator $L$ of order
$2m$ on $\Omega$, and equip the function space
$H = \operatorname{dom}(L^{1/2})$ with the inner product

$$
\langle u, v \rangle_H = \langle Lu, v \rangle_{L^2}.
$$

If $m > d/2$, then $H$ embeds continuously into $C^0(\overline{\Omega})$ and
thus is an RKHS. The inverse operator $L^{-1}$ has a symmetric integral kernel
$G(x,y)$ satisfying

$$
L_x G(x,y) = \delta(x-y), \qquad G(x,y)=G(y,x).
$$

For any $u \in H$ and fixed $y \in \Omega$, integration by parts gives

$$
u(y) = (L^{-1} Lu)(y) = \int_\Omega G(x,y) (Lu)(x)\,dx = \langle u, G(\cdot,y) \rangle_H.
$$

Thus $G(\cdot,y)$ is the reproducing kernel representer at $y$, and

$$
K(x,y) = G(x,y).
$$

\subsection{Weighted Sobolev Spaces and Laplace-Transform RKHSs}
For weights of the form $w(t)=t^{2s-1}e^{-\alpha t}$, the reproducing kernel
of $H_w$ is

$$
K(\lambda,\mu) = \frac{\Gamma(2s)}{(\lambda+\mu+\alpha)^{2s}}.
$$

This kernel is the Green's function of the weighted differential operator

$$
L = \left(\alpha - \frac{d}{d\lambda}\right)^{2s}
$$

on $(0,\infty)$, with the weight $e^{-\alpha\lambda}$ built into the inner
product. Thus $H_w$ is naturally interpreted as a weighted Sobolev space of
order $s$ in the Laplace-variable domain.

\subsection{Why This Matters}
The RKHS/Green's-function viewpoint explains why the dual norm
$\|\mu\|_{H'_w}^2$ can be written as a squared heat-trace integral. It also
connects our Laplace-transform framework to classical Sobolev spaces and to the
Green's-function structure of elliptic operators.

The specific example in the paper shows that the Laplace-transform RKHS
$H_w$ is not an abstract reproducing kernel space, but one whose kernel has a
concrete differential-operator interpretation.

%% file: appendix-randomized-trace-estimation.tex
\section{Randomized Trace Estimation for the Matrix Dual Norm}
\label{app:matrix-dual-norm}

This appendix records the randomized trace-estimation component of the
eigenvalue-free strategy used to approximate the discrete dual norm from the
finite element stiffness and mass matrices.

For conforming finite elements the discrete eigenvalue problem is

$$
K u = \lambda_h M u.
$$

The main subsection explains how the heat trace is reduced to shifted traces
of the form $\tr((K + \sigma M)^{-1}M)$. The point of this appendix is how to
evaluate those traces without forming eigenpairs.

For symmetric $A$, one has the unbiased identity \cite{Hutchinson1990,AvronToledo2011}

$$
\operatorname{tr}(A) = \mathbb{E}[\xi^T A \xi],
$$

and the following variance identity clarifies the probe choice.

\begin{theorem}[Symmetric Hutchinson Trace Estimator and Optimal Probes]
Let $A\in\mathbb R^{n\times n}$ be symmetric and define
$X:=\xi^\top A\xi$, where $\xi=(\xi_1,\dots,\xi_n)^\top$ has i.i.d. entries with
$\mathbb E[\xi_i]=0$, $\mathbb E[\xi_i^2]=1$, and
$\mathbb E[\xi_i^4]=s_4<\infty$. Then
\begin{equation}
\mathbb E[X]=\operatorname{tr}(A),
\qquad
\operatorname{Var}(X)=2\|A\|_F^2+(s_4-3)\sum_{i=1}^n A_{ii}^2.
\end{equation}
Since $s_4\ge 1$, this implies
\[
\operatorname{Var}(X)\ge 2\sum_{i\ne j}A_{ij}^2,
\]
with equality iff $s_4=1$, i.e., iff the probe is Rademacher
($\mathbb P(\xi_i=\pm1)=\tfrac12$). Thus, among i.i.d. mean-zero, variance-one
probes with finite fourth moment, Rademacher probes minimize single-sample variance.
\end{theorem}

In practice one uses the Monte Carlo estimator

$$
\widehat{\operatorname{tr}}(A) = \frac{1}{m}\sum_{r=1}^m \xi_r^T A \xi_r,
$$

with independent probes $\xi_1,\dots,\xi_m$.

Applied to the shifted FEM trace, this gives

$$
\widehat\tau(\sigma) = \frac{1}{m}\sum_{r=1}^m \xi_r^T (K + \sigma M)^{-1} M \xi_r,
$$

so each sample requires one solve with the shifted matrix $K + \sigma M$ and
one multiplication by $M$.

The resulting algorithm is:

\begin{enumerate}
  \item choose the shifts $\sigma_j(t_i)$ that arise from the rational heat-kernel
    approximation in the main text;
  \item for each shift, estimate $\operatorname{tr}((K + \sigma_j(t_i) M)^{-1}M)$ using
    $m$ random probe vectors;
  \item reuse the same probe vectors across shifts when possible so that the
    trace estimates are correlated and the cost is dominated by the linear solves;
  \item combine the trace estimates to form $\widehat Z_h(t_i)$ and then insert
    them into the outer quadrature for $\|\mu_h\|_{H'_w}^2$.
\end{enumerate}

The method avoids full diagonalization of the generalized eigenproblem and
relies only on shifted sparse solves with $K+\sigma M$ and mass-matrix
multiplications. In large problems one can exploit multi-shift solvers and a
small number of randomized probe vectors to estimate the traces efficiently.

Practical remarks:

\begin{itemize}
  \item The trace estimator is unbiased, and
  $\operatorname{Var}(\widehat{\operatorname{tr}}(A))=\operatorname{Var}(X)/m$.
  Hence the variance decreases like $m^{-1}$, and the RMS error like $m^{-1/2}$.
  \item By the theorem above, Rademacher probes are variance-optimal in this i.i.d.
  probe class for symmetric $A$.
  \item Using the same probes across shifts typically reduces noise in the final
  quadrature sum.
  \item The outer quadrature error, the rational approximation error, and the
  trace estimation error all contribute to the final norm approximation error.
\end{itemize}

\subsection{Convergence and Error Estimate}

For each fixed shift $\sigma$ and quadrature node $t_i$, the estimator converges
at the usual Monte Carlo rate. From the theorem above,
$\widehat{\operatorname{tr}}(A)=\frac1m\sum_{r=1}^m X_r$ with i.i.d.
$X_r=\xi_r^\top A\xi_r$ satisfies

$$
\mathbb{E}\bigl[|\widehat{\operatorname{tr}}(A)-\operatorname{tr}(A)|^2\bigr]
= \frac{1}{m}\Bigl(2\|A\|_F^2 + (s_4-3)\sum_{i=1}^n A_{ii}^2\Bigr),
$$

so the root-mean-square error is $O(m^{-1/2})$ and the fluctuations decay like
$m^{-1/2}$ as the number of probes $m$ grows.
Applied to $A=(K+\sigma M)^{-1}M$, this gives

$$
\widehat\tau(\sigma)-\tau(\sigma)=O_{\mathbb P}(m^{-1/2}),
$$

up to a prefactor depending on the size of the shifted inverse.

For the full approximation of $\|\mu_h\|_{H'_w}^2$, the dominant errors are
the outer quadrature error, the rational-approximation error for the heat
kernel, and the stochastic trace-estimation error. A convenient way to view the
result is

$$
\bigl|\widehat{\|\mu_h\|_{H'_w}^2}-\|\mu_h\|_{H'_w}^2\bigr|
= O\bigl(\varepsilon_{\mathrm{quad}} + \varepsilon_{\mathrm{rat}} +
\varepsilon_{\mathrm{tr}}\bigr),
$$

with $\varepsilon_{\mathrm{tr}} = O_{\mathbb P}(m^{-1/2})$.
Thus the method converges as the quadrature is refined, the rational
approximation is improved, and the number of random probes is increased.

This appendix therefore provides the self-contained matrix-based derivation
that complements the main numerical verification section.

%% file: appendix-rational-approximation-gradient.tex
\section{Closed Forms and Gradients for Pole Optimization}
\label{app:rational-approximation-gradient}

The quadratic-form reduction in Proposition~\ref{prop:quadratic-form-elimination} eliminates the
linear coefficients for fixed poles and leaves a reduced objective in the pole parameters.
For the model weight used in the numerical optimization, the resulting matrix entries and gradient
are available in closed form.
This explicit gradient is important in practice because it enables efficient gradient-based
optimization algorithms for locating the poles without repeated numerical quadrature of the
objective or its derivatives.

\begin{proposition}[Model-weight closed forms and reduced gradient]
    \label{prop:model-weight-closed-forms}
    Let $\kappa>0$ and let $\omega(t)=t^\alpha e^{-\delta t}$ with $\delta\geq0$.
    Assume $\alpha<1$, $\alpha<\beta$, $\alpha<2\beta-1$, and let the $\gamma_j>0$ be pairwise distinct
    for $j=1,\dots,n$.
    Then the Gram matrix $A(\gamma)$ is invertible, and we set $u(\gamma):=A(\gamma)^{-1}b(\gamma)$.

    Then:
    \begin{enumerate}
        \item The quantities in the quadratic form satisfy
              \[
                  a=\Gamma(2\beta-1-\alpha)\,(2\kappa^2+\delta)^{-(2\beta-1-\alpha)},
              \]
              \[
                  b_j(\gamma)=\Gamma(\beta-\alpha)(2\kappa^2+\gamma_j+\delta)^{-(\beta-\alpha)},
                  \qquad
                  A_{jk}(\gamma)=\Gamma(1-\alpha)(2\kappa^2+\gamma_j+\gamma_k+\delta)^{-(1-\alpha)}.
              \]

        \item The reduced objective has gradient
              \[
                  \partial_{\gamma_i}\widehat{\mathcal{J}}_n(\gamma)
                  =-2\,(\partial_{\gamma_i}b(\gamma))^T u(\gamma)
                  +u(\gamma)^T(\partial_{\gamma_i}
                  A(\gamma))u(\gamma).
              \]
              Hence
              \[
                  \begin{aligned}
                      \partial_{\gamma_i}\widehat{\mathcal{J}}_n(\gamma)
                       & =2(\beta-\alpha)\Gamma(\beta-\alpha)
                      (2\kappa^2+\gamma_i+\delta)^{-(\beta-\alpha+1)}u_i \\
                       & \quad-2(1-\alpha)\Gamma(1-\alpha)u_i
                      \sum_{k=1}^n u_k(2\kappa^2+\gamma_i+\gamma_k+\delta)^{-(2-\alpha)}.
                  \end{aligned}
              \]
    \end{enumerate}
\end{proposition}

\begin{proof}
    The entries of the quadratic form follow from \[ \int_0^\infty t^{p-1}e^{-qt}\,dt=\Gamma(p)q^{-p},
        \qquad p,q>0, \] applied with \[ (p,q)=(2\beta-1-\alpha,2\kappa^2+\delta),\quad
        (p,q)=(\beta-\alpha,2\kappa^2+\gamma_j+\delta),\quad
        (p,q)=(1-\alpha,2\kappa^2+\gamma_j+\gamma_k+\delta).
    \]
    The hypotheses ensure that all these parameters are positive, giving the
    formulas for $a$, $b_j(\gamma)$, and $A_{jk}(\gamma)$ in part~(1).

    For part~(2), set $u=A^{-1}b$.
    Since $A$ is symmetric, differentiation of $A^{-1}A=I$ gives \[ \partial_{\gamma_i}(A^{-1})
        =-A^{-1}(\partial_{\gamma_i}A)A^{-1}.
    \]
    Therefore
    \[
        \partial_{\gamma_i}(b^T A^{-1}b)
        =2(\partial_{\gamma_i}b)^T u
        -u^T(\partial_{\gamma_i}
        A)u, \] and hence \[ \partial_{\gamma_i}\widehat{\mathcal{J}}_n =-2(\partial_{\gamma_i}b)^T u
        +u^T(\partial_{\gamma_i}A)u.
    \]
    Differentiating the formulas in part~(1) yields
    \[
        \partial_{\gamma_i}b_i
        =-(\beta-\alpha)\Gamma(\beta-\alpha)
        (2\kappa^2+\gamma_i+\delta)^{-(\beta-\alpha+1)}
    \]
    and
    \[
        \partial_{\gamma_i}
        A_{ik} =-(1-\alpha)\Gamma(1-\alpha) (2\kappa^2+\gamma_i+\gamma_k+\delta)^{-(2-\alpha)}.
    \]
    The other components of $\partial_{\gamma_i}b$ vanish.
    Since $A$ is symmetric, the derivative of the quadratic form includes both the $i$th row and
    column, producing the displayed factor of $2$ and the stated componentwise gradient formula.
\end{proof}

%% file: appendix-tensor-product-fem.tex
\section{Tensor-product FEM Spectral Construction}
\label{app:tensor-product-fem}

This appendix records the tensor-product derivation used in the numerical
experiments on aligned Cartesian meshes.

For aligned Cartesian meshes on the square and cube, the finite element
construction is a tensor product of 1D spaces. In 2D with bilinear
$Q_1$ quadrilaterals, the stiffness and mass matrices factor as
\[
K = K_x \otimes M_y + M_x \otimes K_y,
\qquad
M = M_x \otimes M_y.
\]
In 3D with trilinear $Q_1$ hexahedra, the analogous factorization includes
$z$-direction contributions.

When the global space is a tensor product of 1D spaces, the generalized
problem reduces to lower-dimensional eigenproblems:
\[
K_x u_i = \mu_i M_x u_i,
\qquad
K_y v_j = \nu_j M_y v_j,
\qquad
K_z w_k = \xi_k M_z w_k.
\]
The discrete eigenvalues are then additive combinations of 1D eigenvalues:
\[
\lambda_{ij} = \mu_i + \nu_j \qquad \text{in 2D},
\]
\[
\lambda_{ijk} = \mu_i + \nu_j + \xi_k \qquad \text{in 3D}.
\]

For heat-trace based quantities this also yields separable formulas. In 2D,
\[
Z_h(t) = \sum_{i,j} e^{-(\mu_i+\nu_j)t}
= \left(\sum_i e^{-\mu_i t}\right)\left(\sum_j e^{-\nu_j t}\right),
\]
and in 3D,
\[
Z_h(t)
= \sum_{i,j,k} e^{-(\mu_i+\nu_j+\xi_k)t}
= \left(\sum_i e^{-\mu_i t}\right)
\left(\sum_j e^{-\nu_j t}\right)
\left(\sum_k e^{-\xi_k t}\right).
\]

This construction provides an inexpensive way to assemble aligned-mesh spectra,
to check separability numerically, and to compare full global solves with
1D-based decompositions.